\documentclass[a4paper,12pt]{amsart}
\usepackage{amssymb}
\usepackage{amsmath, amsthm, amscd, amsfonts, amssymb, graphicx, color}
\usepackage{amsmath, amsthm, amscd, amsfonts, amssymb, graphicx, color}

\usepackage[cp1250]{inputenc}
\usepackage[T1]{fontenc}

\newtheorem{theorem}{Theorem}[section]
\newtheorem{lemma}[theorem]{Lemma}
\newtheorem{proposition}[theorem]{Proposition}
\newtheorem{corollary}[theorem]{Corollary}
\theoremstyle{definition}
\newtheorem{definition}[theorem]{Definition}
\newtheorem{example}[theorem]{Example}

\theoremstyle{remark}
\newtheorem{remark}[theorem]{Remark}
\numberwithin{equation}{section}

\begin{document}

	\vspace{15pt}
	
	\title{Periodicity and Devaney chaos in Banach algebras}
	\author{Stefan Ivkovi\' c}
	\address{Mathematical Institute of the Serbian Academy of Sciences and Arts, Kneza Mihaila 36, Beograd
		11000, Serbia}
	\email{stefan.iv10@outlook.com}
	\maketitle
	
		\begin{abstract}
		In this paper, we consider operators that are compositions of an isometric isomorphism and a left multiplier on a Banach algebra, and we provide necessary and sufficient conditions for these operators to have dense set of periodic elements. As an application of this result, we characterize generalized weighted shifts with a dense set of periodic elements on the standard Hilbert module over C*-algebra of compact operators on a separable Hilbert space. 
	\end{abstract}
	
	
	\vspace{15pt}
	
	\begin{flushleft}
	\end{flushleft}
	
	\

	\section{Introduction}
	
	Linear dynamics of operators have been studied in many articles during several decades; see \cite{Bayart} and \cite{grosse-erdmann-chaos-2} as monographs on this topic. Topological transitivity and chaos are central topics in this field. We recall the definition of these concepts. 
	
	\begin{definition}
		Let $\mathcal{X}$ be a Banach space and $ B(\mathcal{X})$ be the space of all bounded linear operators on $ \mathcal X.$ A sequence of operators  $ \lbrace T_n \rbrace_{n\in \mathbb{N}_0} \subseteq B(\mathcal{X})$ is said to be {\it topologically transitive} on $\mathcal{X}$ if for each pair of open non-empty subsets $O_{1}$ and $O_{2}$ of $\mathcal{X}$ there exists some $n \in \mathbb{N}$ such that $  T_{n}(O_{1}) \cap O_{2} \neq \varnothing .$ If for each pair of open non-empty subsets $O_{1}$ and $O_{2}$ of $\mathcal{X}$ there exists some $N \in \mathbb{N}$ such that $  T_{n}(O_{1}) \cap O_{2} \neq \varnothing $ for all $ n \geq N ,$ then $ \lbrace T_n \rbrace_{n\in \mathbb{N}_0}$ is said to be \textit{topologically mixing}. An operator $T\in B(\mathcal{X})$ is said to be topologically transitive (respectively topologically mixing) if the sequence $\{T^n\}_{n\in \mathbb{N}_0}$ is topologically transitive (respectively topologically mixing).
	\end{definition}
	
	\begin{definition}\label{prva} 
		Let $\mathcal X$ be a Banach space, and $(T_n)_{n\in \mathbb{N}_0}$  be a sequence of operators in $B(\mathcal X)$, where $B(\mathcal X)$ is the space of all bounded linear operators on $ \mathcal X $ equipped with the operator norm. A vector $x\in \mathcal X$ is called a {\it periodic element} of $(T_n)_{n\in \mathbb{N}_0}$ if there exists a constant $N\in\mathbb N$ such that for each $k\in\mathbb N$, $T_{kN}x=x$. The set of all periodic elements of $(T_n)_{n\in \mathbb{N}_0}$ is denoted by
		${\mathcal P}((T_n)_{n\in \mathbb{N}_0})$. The sequence $(T_n)_{n\in \mathbb{N}_0}$ is called {\it Devaney chaotic} if $(T_n)_{n\in \mathbb{N}_0}$ is topologically transitive and ${\mathcal P}((T_n)_{n\in \mathbb{N}_0})$ is dense in ${\mathcal X}$. An operator $T\in B(\mathcal{X})$ is called \emph{Devaney chaotic} if the sequence $\{T^n\}_{n\in \mathbb{N}_0}$ is Devaney chaotic. 
\end{definition}
Devaney chaos is related to the description of dynamical systems that are sensitive to initial conditions.  We refer to \cite{chaos-bonet-2, devaney, grosse-erdmann-chaos-2, chaos-herrero, chaos-tensor-1, chaos-skew, chaos-revisited} for the general chaos theory. Since groundbreaking papers by Birkhoff \cite{birkhof}  and Godefroy and Shapiro \cite{godefroy-chaos}, an extensive literature on the properties and characterization of chaotic operators have been published, see \cite{semi-chaotic, chaos-bonet, chaotic-multipliers, Chen2014, chaos-frequent, grosse-erdmann-chaos-1, chaos-shift, advances, topol-chaos}. Moreover, the theory of Devaney chaos has applications in difference and partial differential equations \cite{chaos-equations, chaos-differential-3, chaos-differential-1, chaos-differential-2}, as well in the study of operator semigroups \cite{kalmes2007}. As it follows from Definition \ref{prva}, Devaney chaos is connected to the density of periodic elements. The study of the properties and of the structure of the set of periodic and almost periodic elements is an active research topic, we refer to \cite{otmane} for some recent results in this field.

On the other hand, linear dynamics of operators on Banach algebras and their ideals
constitutes an active area of research; see, for instance,
\cite{MDPI, chaotic-multipliers,ZZD,BIMS,AOFA,iran1,kina} and the references therein. For example, a characterization of universal and chaotic multipliers on some ideals of operators, in particular on the ideal of compact operators on a Banach space with approximation property, has been provided in \cite{chaotic-multipliers}, see also \cite{ZZD}. Hypercyclicity and supercyclicity of some classes of elementary operators on the ideal of compact operators have been investigated in for instance \cite{gupta, pet}. In addition, new dynamical concepts such as  C*-transitivity, C*-numerical transitivity and generalized numerical transitivity for left multipliers on the ideal of compact operators and for diagonal operators on the standard Hilbert C*-module have been introduced and explored in \cite {MDPI}. Finally, Furstenberg semi-transitivity of generalized weighted shifts on the standard Hilbert C*-module and on more general non-unital Banach algebras has been studied in \cite{filomat,AOFA}. More precisely, for a non-unital Banach algebra $\mathcal A $ which is an ideal in a unital Banach algebra $ \mathcal A_1,$ in \cite{AOFA} the author considered the operators on $\mathcal A $ of the form $ a \to b \Phi(a) ,$ where $b$ is an invertible element in $ \mathcal A_1$ and $ \Phi$ is an isometric isomorphism on $ \mathcal A, $ and obtained necessary and sufficient conditions for these operators to be Furstenberg semi-transitive on $ \mathcal A.$ Then, as an application and a special case of this theory, a characterization of Furstenberg-semi-transitive generalized weighted shifts on the standard Hilbert module  $l_{2}(\mathcal{C})$ over a C*-algebra has been deduced in \cite{AOFA}.

Motivated by this research-line and based on the approach from \cite{AOFA}, in the present paper we go a step further and characterize a subclass of the operators $ a \to b \Phi(a) $ whose set of periodic elements is dense in $ \mathcal A$ (the notation here is the same as above, adopted from \cite{AOFA}). In Theorem \ref{glavno-haos} we provide such characterization. As expected, the purpose of this result is to be applied in the study of Devaney chaotic generalized weighted shifts on  $l_{2}(\mathcal{C}).$ Our first application of Theorem \ref{glavno-haos} is in the situation of $l_{2}(\mathcal{C})$ when $ \mathcal C$ is the C*-algebra of compact operators on a separable Hilbert space $H.$ More precisely, in this application we let $ \lbrace e_j \rbrace _{j \in \mathbb Z} $ be an orthonormal basis for $H$ and  $ \mathcal U$ be the unitary operator on $H$ satisfying that $ \mathcal U(e_j)= e _{j+1} $ for all $ j \in \mathbb Z.$ Then, we consider generalized weighted shifts $l_{2}(\mathcal{C})$ of the form $$ \lbrace F_j \rbrace_{j \in \mathbb Z} \to  \lbrace W_j F_{j-1} \mathcal U \rbrace_{j \in \mathbb Z} ,$$ where $ \lbrace W_j \rbrace_{j \in \mathbb Z}$ is a sequence of bounded invertible linear  operators on the respective Hilbert space $H,$ and  in Theorem \ref{period-kompakt}, which is the main result of this paper, we completely characterize Devvaney chaotic such shifts in terms of the weight sequence $ \lbrace W_j \rbrace_{j \in \mathbb Z}.$ This characterization follows as a consequence of Theorem \ref{glavno-haos} which is stated in the more abstract setting of operators on non-unital Banach algebras.

We would like to mention in this context that the dynamics of generalized weighted shifts with operator-valued weights seem to be an active international research topic, see for instance \cite{iv2, FIL, filomat, advances, Nkoreja}. Moreover, since pioneering papers by W. Paschke and M. Rieffel \cite{rieffel}, numerous papers on operator theory in Hilbert C*-modules have been published, see \cite{MT} and the references therein. All these facts were additional motivating reasons for us to investigate Devaney chaos of generalized weighted shifts on $l_{2}(\mathcal{C}).$\\
 Now, in \cite[Theorem 3.3]{Nkoreja} Ri , Ju and Kim considered another kind of generalized weighted  shifts on $l_{2}(\mathcal{C})$ given by $ \lbrace F_j \rbrace_{j \in \mathbb Z} \to  \lbrace W_j F_{j-1} \rbrace_{j \in \mathbb Z} ,$ 
and they obtained equivalent conditions for the Devaney chaos of these shifts in terms of the respective weight sequence. However, Devaney of more general weighted shifts on $l_{2}(\mathcal{C})$ of the form $ \lbrace F_j \rbrace_{j \in \mathbb Z} \to  \lbrace W_j F_{j-1} V \rbrace_{j \in \mathbb Z} $ with $V$ being a general unitary operator on $H$ have not been studied in the literature so far, although several other dynamical properties like topological transitivity and Furstenberg transitivity for this more general class of shifts have been investigated in detail in \cite{FIL, filomat, Nkoreja}. In particular, it has been proved in \cite[Proposition 3.1]{FIL} that topological transitivity of the the weighted shifts of the form $ \lbrace F_j \rbrace_{j \in \mathbb Z} \to  \lbrace W_j F_{j-1} V \rbrace_{j \in \mathbb Z} $ is independent of the choice of the unitary operator $V,$ i.e. the equivalent conditions (given in terms of the respective weight sequence) for topological transitivity of the generalized weighted shift $ \lbrace F_j \rbrace_{j \in \mathbb Z} \to  \lbrace W_j F_{j-1} V \rbrace_{j \in \mathbb Z} $ are the same no matter which unitary operator $U$ is considered. In other words,  $$ \lbrace F_j \rbrace_{j \in \mathbb Z} \to  \lbrace W_j F_{j-1} \rbrace_{j \in \mathbb Z} $$  is topologically transitive if and only if  $$ \lbrace F_j \rbrace_{j \in \mathbb Z} \to  \lbrace W_j F_{j-1} V \rbrace_{j \in \mathbb Z} $$ is topologically transitive for every unitary operator $V$ on $H.$ A natural question which arises in this setting is therefore whether a similar correspondence in the dynamical behaviour of these shifts remains valid in the context of Devaney chaos. That is, we arrive to the following question:

\text{ } 

\textbf{Question 1:} What are the relations between the conditions for the Devaney chaos of the operator $ \lbrace F_j \rbrace_{j \in \mathbb Z} \to  \lbrace W_j F_{j-1} \mathcal U \rbrace_{j \in \mathbb Z} $ given in Theorem \ref{period-kompakt} in this paper and the conditions for the Devaney chaos of the operator  $$ \lbrace F_j \rbrace_{j \in \mathbb Z} \to  \lbrace W_j F_{j-1} \rbrace_{j \in \mathbb Z} $$ given in \cite[Theorem 3.3]{Nkoreja}?
 
In particular, does the Devaney chaos of the operator $ \lbrace F_j \rbrace_{j \in \mathbb Z} \to  \lbrace W_j F_{j-1} \rbrace_{j \in \mathbb Z} $ imply the Devaney chaos of $ \lbrace F_j \rbrace_{j \in \mathbb Z} \to  \lbrace W_j F_{j-1} \mathcal U \rbrace_{j \in \mathbb Z} $ and/or vice versa?

\text{ }

To simplify notation, we denote by $T$ the operator $$ \lbrace F_j \rbrace_{j \in \mathbb Z} \to  \lbrace W_j F_{j-1} \mathcal U \rbrace_{j \in \mathbb Z} ,$$ and we denote by $L$ the operator $$ \lbrace F_j \rbrace_{j \in \mathbb Z} \to  \lbrace W_j F_{j-1} \rbrace_{j \in \mathbb Z} .$$ Thanks to Theorem \ref{period-kompakt} in this paper, we prove that the Devaney chaos of the operator $L$ always implies the Devaney chaos of the operator $T,$ whereas the converse does not hold in general. Here it is of course understood that the operators $T$ and $L$ under consideration have the same weight sequence $ \lbrace W_j \rbrace_{j \in \mathbb Z}.$ 

 More precisely, in the paper we derive the above-mentioned answer to Question 1 step by step, as follws:\\
  First, in Example \ref{prvi-primer} we construct an operator sequence $ \lbrace W_j \rbrace_{j \in \mathbb Z}$ which satisfies the conditions of  Theorem \ref{period-kompakt} but which does not satisfy the conditions of \cite[Theorem 3.3]{Nkoreja}. Hence, from  Theorem \ref{period-kompakt} we deduce then that the respective generalized weighted  shift operator $ T $ is Devaney chaotic on $l_{2}(\mathcal{C}),$ however, at a same time,  the induced generalized weighted shift operator $ L $  is \underline{not} Devaney chaotic on $l_{2}(\mathcal{C}),$ (which follows from \cite[Theorem 3.3]{Nkoreja}). This shows that the Devaney chaos of the operator $T$ does not imply in general the Devaney chaos of the operator $ L $ even if $T$ and $L$ share the same weight sequence.\\
  Conversely, in Corollary \ref{poredjenje} we prove that the Devaney chaos of the generalized weighted shift $L$ always implies the Devaney chaos of the generalized weighted shift $ T$ when $T$ and $L$ share the same weight sequence. \\
Thus, the conditions of Theorem \ref{period-kompakt} in this paper for the Devaney chaos of the operator $T$  are strictly weaker than the conditions of \cite[Theorem 3.3]{Nkoreja} for the Devaney chaos of the operator $L$.

\text{ }

Further, in Remark \ref{napomena-shift} we explain the important difference between the generalized weighted shifts considered in \cite{advances} and the generalized weighted shifts on $l_{2}(\mathcal{C}).$ In fact, the norm on $l_{2}(\mathcal{C})$ and the norm on $\ell_2 $ sums in \cite{advances} are different, hence, since dynamics involves the aspects such as density, convergence and topological transitivity, the study of dynamical behaviour of operators on $l_{2}(\mathcal{C})$ and operators on $\ell_2 $ sums in \cite{advances} is essentially different. 

Now, in \cite{advances} they also explored the case when all the weights of the generalized weighted shift operator are equal to one fixed operator $T $ and they studied then the relations between dynamical properties of this operator $T$ and of the induced generalized weighted shift operator. Motivated by this research, in Example \ref{jedan-operator} we investigate the case when all $W_j$`-s are equal to one fixed invertible bounded linear operator $W$ on $H.$ In this example we discover that $W$ can happen to not 
be Devaney chaotic on $H$ and that at a same time the respective generalized weighted shift $ \lbrace F_j \rbrace_{j \in \mathbb Z} \to  \lbrace W F_{j-1} \mathcal U \rbrace_{j \in \mathbb Z} $ is Devaney chaotic on $l_{2}(\mathcal{C}).$

\text{ }

Question 1  inspired us to consider similar questions for the operators on $\mathcal{C} $ of the form $F \to WF \mathcal U^* ,$ where $W$ is an invertible bounded linear operator on $H.$ Here $\mathcal C$ is still the C*-algebra of compact operators on $H$ and $ \mathcal U$ is still the unitary shift operator on $H$ as defined above. Topological transitivity of these operators on $\mathcal C$ has been studied in for instance \cite{iv2, BIMS}, whereas hypercyclicity and supercyclicity of operators of the form $F \to TFS,$ with $T, \text{ } S$ being general bounded linear operators on $H,$ has been investigated in \cite{gupta, pet}. These results motivated us to study in this paper more deeply the Devaney chaos of the operator $F \to WF \mathcal U^* .$  In fact, since $ \mathcal U^*$ acts as a shift with respect to a fixed orthonormal basis of $H$ and $W$ can be considered as a sort of "weight", the operator $F \to WF \mathcal U^* $ can therefore be interpreted as a noncommutative analogue on $ \mathcal C$ of the classical weighted shift operator, and this is also one of the reasons why we are interested in the study of dynamical behaviour of such operator.\\
 Now, in \cite[Corollary 2.4]{BIMS} it has been proved that $F \to WF \mathcal U^* $ is topologically transitive on $\mathcal{C} $ if and only if $W$ satisfies the so-called Hypercyclicity Criterion on $H.$ In particular, this means that $W$ must be topologically transitive on $H$ in this case. Thus, topological transitivity of the operator $F \to WF \mathcal U^* $ generates automatically the requirement on the topological transitivity of the operator $W$ itself.  Further, it is well known from the famous result in linear dynamics  \cite[Theorem 2.1]{chaotic-multipliers} that the left multiplier by $W$ on $ \mathcal C$ (that is the operator of the form $F \to WF $)
  is Devaney chaotic on $ \mathcal C$ if and only if $W$ is Devaney chaootic on $H.$ 
 All these facts led us naturally to the following question.

\text{ } 

\textbf{Question 2:} Does the Devaney chaos of the operator $F \to WF \mathcal U^* $ on $ \mathcal C$ imply the Devaney chaos of $W$ on $H$ ?

\text{ }
 
 Surprisingly, in Proposition \ref{elementarna-preslikavanja} we derive the negative answer to this question. In this proposition, we provide some sufficient conditions for the Devaney chaos of the operator $F \to WF \mathcal U^* $ on $ \mathcal C$ and then we construct a bounded invertible operator $W$ on $H$ which satisfies these conditions but which is not Devaney chaotic on $H.$ Thus, the induced operator $F \to WF \mathcal U^* $ is then Devaney chaotic on $ \mathcal C$ althoough $W$ itself is not Devaney chaotic on $H.$\\
  The proof of Proposition \ref{elementarna-preslikavanja} is also based on an application of Theorem \ref{glavno-haos}. Further, in Remark \ref{veza-Shiraz} we argue that the sufficient conditions of Proposition  \ref{elementarna-preslikavanja} differ from the conditions in \cite[Theorem 5]{iv2}.
  
  Finally, we end the paper with the open question whether the Devaney chaos of the operator  $W$ on $H$ implies the Devaney chaos of the induced operator $F \to WF \mathcal U^* $ on $ \mathcal C.$

\section{Preliminaries}

	 Throughout the paper, we let $ \mathcal{A}$ be a non-unital Banach algebra which is a left ideal in a unital Banach algebra $\mathcal A_{1} .$ We will assume that there exists a constant $ K> 0$ such that  $$ \| ba \| \leq K \| b \|_1 \| a \|, $$for all $ a \in \mathcal{A}$ and $ b \in \mathcal{A}_1 ,$ where $ \| \cdot \| $ and $ \| \cdot \|_1 $ denote the norms on $\mathcal{A} $ and $\mathcal{A}_1 ,$ respectively. We do not assume that the norm on $\mathcal{A} $ extends to the norm on  $\mathcal{A}_1 .$ This condition will be called the \textit{condition (E)} throughout the paper.
	
	Let $\{p_{\alpha}\}_{\alpha}$ be a set in $\mathcal A.$	We will assume that $\{p_{\alpha}\}_{\alpha}$
	has \textit{right-approximation property} in $\mathcal{A},$ that is for every open non-empty subset $\mathcal{O}$ of
	$\mathcal{A}$ and $u \in \mathcal{O}$ there exists some $\alpha$ such
	that
	$
	u p_{\alpha} \in \mathcal{O}.
	$ Also, we will assume that for each $a \in  \mathcal{A}$ and every index $ \alpha $ it holds that
	$
	\|ap_{\alpha}\|\leq \|a\|
	.$ This condition will be called the \textit{condition (L)} throughout the paper.

	Next, we will assume that $\Phi$ is an isometric algebra isomorphism of $\mathcal{A}_1 $ (in the norm $ \| \cdot \|_1 $) such that the restriction of $\Phi$ to $\mathcal{A} $ is an isometric isomorphism of $\mathcal{A} $ in the norm $ \| \cdot \| .$  We will say that $\Phi$ is \textit{strongly aperiodic with respect to }$\{p_{\alpha}\}_{\alpha}$ if
	for every $\alpha$ there exists some $N_{\alpha}\in\mathbb N$
	such that for all $a\in\mathcal A$ and every $l\geq N_{\alpha}$,
	the series
	\[
	\sum_{n=1}^{\infty}
	a\Phi^{nl}(p_{\alpha})
	\qquad\text{and}\qquad
	\sum_{n=1}^{\infty}
	a\Phi^{-nl}(p_{\alpha})
	\]
	are convergent in $\mathcal A$, and in addition
	\[
	\lim_{k\to\infty}
	\sum_{n=1}^{\infty}
	a\Phi^{nl k}(p_{\alpha})
	=
	\lim_{k\to\infty}
	\sum_{n=1}^{\infty}
	a\Phi^{-nl k}(p_{\alpha})
	=
	0.
	\]

	Let $b$ be an invertible element in $\mathcal A_{1}$. We define the operator $T$ on $\mathcal A$ by $$ T(a)= b \Phi(a)$$ for all $a \in \mathcal A.$ Due to the condition (E), the operator $T$ is bounded linear operator. Moreover, since $ \Phi $ is an isometric isomorphism and $b$ is invertible, it follows that the operator $T$ is invertible. We set $ S:=T^{-1} .$ By some calculations it can be checked that
	\[
	T^{n}(a)
	=
	b\,
	\Phi(b)
	\dots
	\Phi^{\,n-1}(b)\,
	\Phi^{n}(a),
	\]
	and
	\[
	S^{n}(a)
	=
	\Phi^{-1}(b^{-1})
	\dots
	\Phi^{-n}(b^{-1})
	\Phi^{-n}(a),
	\]
	for all $a\in\mathcal{A}$.\\
	Disjoint Furstenberg semi-transitivity of such operators has been studied in \cite{AOFA}. 
	
	\text{ }
	
	Let $\beta$ be a subset of
	$\mathcal A_{1}\cap G(\mathcal A_{1})$,
	where $G(\mathcal A_{1})$ is the set of all invertible
	elements of $\mathcal A_{1}$. For every $\widetilde{b}\in\beta$, we let
	$\widetilde{b}T$ be the operator on $\mathcal A$
	defined by
	
	\[
	(\widetilde{b}T)(a)
	=
	\widetilde{b}\,T(a).
	\]
	for all $a\in\mathcal A$. An element $x\in\mathcal A$
	is said to be a $\beta$-\textit{periodic element} of $T$ if there
	exists some $\widetilde{b}\in\beta$ and some
	$d\in\mathbb N$ such that	
	\[
	(\widetilde{b}T)^{nd}(x)=x
	\]
	for all $n\in\mathbb N$.
	If \[
	(\widetilde{b}T)^{nd}(x)=x
	\]
	for all $n\in\mathbb N,$ then $x$ is a just called a \textit{periodic element} of $T.$
	
	\section{Main result}
	We are now ready to present the first of the two main theorems of this paper.
	
	\begin{theorem}\label{glavno-haos}
		Under the above notation and assumptions, the following statements are equivalent.
		
		(1) The set of $\beta$-periodic elements of $T$
		is dense in $\mathcal A$.
		
		(2)There exists a dense set $\mathcal{D}$ in $\mathcal A$ such that for every $x\in \mathcal{D}$ and every index $\alpha$ we can
		find some $\widetilde{b}\in\beta$ and $ d \in \mathbb{N}$
		such that the series
		\[
		\sum_{n=1}^{\infty}
		(\widetilde{b}T)^{ndk}(x)\,
		\Phi^{ndk}(p_{\alpha}) \text{ and  }
		\sum_{n=1}^{\infty}
		(\widetilde{b}S)^{ndk}(x)\,
		\Phi^{-ndk}(p_{\alpha})
		\]
		are convergent for each $k\in\mathbb N$, and moreover
		\[
		\lim_{k\to\infty}
		\sum_{n=1}^{\infty}
		(\widetilde{b}T)^{ndk}(x)\,
		\Phi^{ndk}(p_{\alpha})
		=
		\lim_{k\to\infty}
		\sum_{n=1}^{\infty}
		(\widetilde{b}S)^{ndk}(x)\,
		\Phi^{-ndk}(p_{\alpha})
		=
		0.
		\]
	\end{theorem}	
	\begin{proof}
		We prove $(1) \Rightarrow (2)$ first. Let $\mathcal{D}$ be the set of all $\beta$-periodic elements of $T$. If $ x \in \mathcal{D} ,$ then, given a fixed index $\alpha,$ we can find some $ d \in \mathbb{N} $ and some $\widetilde{b}\in\beta$  such that	
		\[
		(\widetilde{b}T)^{nd}(x)=x
		\]
		for all $n\in\mathbb{N}$ and
		
		\[
		\lim_{l\to\infty}
		\sum_{n=1}^{\infty}
		x\Phi^{ndl}(p_{\alpha})
		=
		\lim_{l\to\infty}
		\sum_{n=1}^{\infty}
		x\Phi^{-ndl}(p_{\alpha})
		=
		0,
		\] where the corresponding series are convergent for each $ l \in \mathbb{N}.$ It follows then that
		\[
		\lim_{l\to\infty}
		\sum_{n=1}^{\infty}
		(\widetilde{b}T)^{ndl}(x)\,
		\Phi^{ndl}(p_{\alpha})
		=
		\lim_{l\to\infty}
		\sum_{n=1}^{\infty}
		(\widetilde{b}S)^{ndl}(x)\,
		\Phi^{-ndl}(p_{\alpha})
		=
		0,
		\]
		where the corresponding series are convergent for each $ l \in \mathbb{N}.$
		Since the set of $\beta$-periodic elements of $T$ is dense
		in $\mathcal{A}$ by the assumption, the implication
		$(1) \Rightarrow (2)$ follows.
		
		Now we prove $(2) \Rightarrow (1).$ Notice first that for each $n, d \in \mathbb{N}, x \in \mathcal A, \tilde b \in \beta$ and every index $ \alpha $ we have
		\begin{align*}
			(\widetilde{b}T)^d(\widetilde{b}T)^{nd}(x)\,
			\Phi^{nd}(p_{\alpha})
			&=
			\widetilde{b}b\Phi(\widetilde{b}b)\dots
			\Phi^{d-1}(\widetilde{b}b)\,
			\Phi^d
			\Big(
			\widetilde{b}b \dots \Phi^{nd-1}(\widetilde{b}b))\Phi^{nd}(xp_{\alpha})
			\Big) \\
			&=
			\widetilde{b}b\dots
			\Phi^{(n+1)d-1}(\widetilde{b}b)\,
			\Phi^{(n+1)d}(x)\,
			\Phi^{(n+1)d}(p_{\alpha}) \\
			&=
			(\widetilde{b}T)^{(n+1)d}(x)\,
			\Phi^{(n+1)d}(p_{\alpha}).
		\end{align*}
		Similarly,
		\begin{align*}
			(\widetilde{b}T)^d \Big( (\widetilde{b}S)^{nd}(x)\,
			\Phi^{-nd}(p_{\alpha}) \Big)
			&=
			\widetilde{b}b \dots
			\Phi^{d-1}(\widetilde{b}b)\,
			\Phi^{d}
			\Big(
			\Phi^{-1}(\widetilde({b}b)^{-1})
			\cdots
			\Phi^{-nd}((\widetilde{b}b)^{-1}
			x p_{\alpha})\Big)
			\\
			&=
			\Phi^{-1}\big((\widetilde{b}b)^{-1})
			\dots
			\Phi^{-(n-1)d}
			\big((\widetilde{b}b)^{-1}
			x  \big)\,
			\Phi^{-(n-1)d}(p_{\alpha})
			\\
			&=
			(\widetilde{b}S)^{(n-1)d}(x)\,
			\Phi^{-(n-1)d}(p_{\alpha}),
		\end{align*}
		
		for each $n\in\mathbb N$.
		
		Hence
		\[
		x p_{\alpha}+
		\sum_{n=1}^{\infty}
		(\widetilde{b}T)^{nd}(x)\,
		\Phi^{nd}(p_{\alpha})
		+
		\sum_{n=1}^{\infty}
		(\widetilde{b}S)^{nd}(x)\,
		\Phi^{-nd}(p_{\alpha})
		\]
		is a periodic element of $\widetilde{b}T$
		with period $d$, whenever the series
		\[
		\sum_{n=1}^{\infty}
		(\widetilde{b}T)^{nd}(x)\,
		\Phi^{nd}(p_{\alpha}) \text{ and }
		\sum_{n=1}^{\infty}
		(\widetilde{b}S)^{-nd}(x)\,
		\Phi^{-nd}(p_{\alpha})
		\]
		converge.
		Now, let $\mathcal{O}$ be an open non-empty subset of
		$\mathcal A$ and $u \in \mathcal{O}$. Then there exists some $\alpha$ such
		that
		$
		u p_{\alpha} \in \mathcal{O}.
		$
		Also, we can find some $x \in \mathcal{D}$ such that
		$$
		\|x-u\|<\frac{\delta}{2},
		$$
		hence
		$$
		\|xp_{\alpha}-up_{\alpha}\|<\frac{\delta}{2},
		$$
		where $\delta>0$ is such that the $\delta$-neighbourhood
		of $up_{\alpha}$ is contained in $\mathcal{O}$.
		Choose $d \in \mathbb{N}$ and $\widetilde{b}\in \beta$ satisfying the assumptions in 2)
		with respect to $x$ and $\alpha$. Since the series
		$$
		\sum_{n=1}^{\infty}
		(\widetilde{b}T)^{ndk}(x)\,
		\Phi^{ndk}(p_{\alpha}) \text{ and }
		\sum_{n=1}^{\infty}
		(\widetilde{b}S)^{ndk}(x)\,
		\Phi^{-ndk}(p_{\alpha})
		$$
		converge in $\mathcal{A}$ for each $k \in \mathbb{N}$, and since
		\[
		\lim_{k\to\infty}
		\sum_{n=1}^{\infty}
		(\widetilde{b}T)^{ndk}(x)\,
		\Phi^{ndk}(p_{\alpha})
		=
		\lim_{k\to\infty}
		\sum_{n=1}^{\infty}
		(\widetilde{b}S)^{ndk}(x)\,
		\Phi^{-ndk}(p_{\alpha})
		=
		0,
		\]
		we can find some $M\in\mathbb{N}$ such that
		\[
		\left\|
		\sum_{n=1}^{\infty}
		(\widetilde{b}T)^{ndM}(x)\,
		\Phi^{ndM}(p_{\alpha})
		\right\|
		<
		\frac{\delta}{4}
		\]
		and
		\[
		\left\|
		\sum_{n=1}^{\infty}
		(\widetilde{b}S)^{ndM}(x)\,
		\Phi^{-ndM}(p_{\alpha})
		\right\|
		<
		\frac{\delta}{4}.
		\]
		
		Let
		\[
		y
		:=
		xp_{\alpha}
		+
		\sum_{n=1}^{\infty}
		(\widetilde{b}T)^{ndM}(x)\,
		\Phi^{ndM}(p_{\alpha})
		+
		\sum_{n=1}^{\infty}
		(\widetilde{b}S)^{ndM}(x)\,
		\Phi^{-ndM}(p_{\alpha}).
		\]
		Then
		$
		\|y-up_{\alpha}\|<\delta,
		$	
		so $y\in\mathcal{O}$. Moreover, by the above arguments,
		$y$ is a periodic element of $\widetilde{b}T$ with period $dM$.
		This shows that the set of $\beta$-periodic elements of $T$ is dense
		in $\mathcal{A}$.
	\end{proof}
	
	\section{Applications}
	From now on and in the rest of this section, for a separable Hilbert space $H$ we let $B(H)$ be  
	the space of all bounded linear  operators on $H$, and  
	$\mathcal{C}$ be the $C^{*}$-algebra of compact operators on $H$. For an orthonormal basis $\{ e_{j} \}_{j \in \mathbb{Z}}$ for $H$,  
	we let for each $m \in \mathbb{N}$, $P_{m}$ be the orthogonal projection  
	onto $\text{Span} \{ e_{-m}, \dots, e_{m} \}$. For every $ J \in \mathbb{N}$ we let 
	$
	[J]:=\{-J,-J+1,\ldots,J-1,J\}.
	$
	Further, we will denote by $l_{2}(\mathcal{C})$  
	the standard (right) Hilbert module over $\mathcal{C}$, see \cite[Example 1.3.5]{MT}. Notice that $l_{2}(\mathcal{C})$ is a non-unital Banach algebra. Indeed, we can define multiplication on  $l_{2}(\mathcal{C})$ as pointwise multiplication, i.e., if $\{x_j\}_{j\in\mathbb{Z}}, \{y_j\}_{j\in\mathbb{Z}} \in l_{2}(\mathcal{C})$,  
	then $$\{x_j \}_{j\in\mathbb{Z}} \cdot \{ y_j\}_{j\in\mathbb{Z}}=\{x_j y_j\}_{j\in\mathbb{Z}} ,$$ for the details we refer to \cite{AOFA}. We put then $\mathcal{A} = l_{2}(\mathcal{C}),$ 
	and $$\mathcal{A}_1 := \{  \{ y_j \}_{j \in \mathbb{Z}} \mid y_j \in B(H) \ \forall j, \, \exists M_y > 0 
	\text{ such that } \| y_j \| \le M_y \ \forall j \} .$$
	The multiplication on $\mathcal{A}_1$ is defined similarly component-wise. We define the norm $ \| \cdot \|_1 $ on $\mathcal{A}_1$ as $ \| \{ y_j \}_{j \in \mathbb{Z}} \|_1 = sup_{j \in \mathbb{Z} } \| y_j \| $ for all $ \{ y_j \}_{j \in \mathbb{Z}} \in \mathcal{A}_1 .$ It is straightforward to check that the condition \textit{(E)} is satisfied in this case.
	
	For each $m, J \in \mathbb{N},$ we let 
	$\tilde{p}_{J,m} \in l_{2}(\mathcal{C})$ be given by
	$$
	(\tilde{p}_{J,m})_{i} =
	\begin{cases}
		P_{m}, & \text{if } -J \le i \le J, \\
		0, & \text{else}.
	\end{cases}
	$$
		It remains to show that
	$\widetilde{p}_{(J,m)}$ is has right approximation  property
	in $\mathcal{A}$. To this end, let
	\[
	b=\{T_{j}\}_{j\in\mathbb Z}\in\mathcal{A}
	\]
	and $\varepsilon>0$. Then we can find some $J\in\mathbb{N}$ such that for all
	\[
	\left\|
	\sum_{j\in\mathbb{Z}\setminus[J]}
	T_j^{*}T_j
	\right\|
	<
	\varepsilon.
	\]
	Since $[J]$ is finite, by \cite[Proposition 2.2.1]{MT} we can find some
	$m\in\mathbb{N}$ such that
	\[
	\|T_jP_{m}\|^{2}
	<
	\frac{\varepsilon}{[J]}
	\]
	for all $j\in[J]$. It is then easy to check that
	\[
	\left\|
	b-b\widetilde{p}_{(J,m)}
	\right\|
	<
	\varepsilon.
	\]			
	Moreover, since \[
	\left\|
	\sum_{j\in [J]}
	P_m T_j^{*}T_j P_m
	\right\|= \left\|P_m
	\bigl(\sum_{j\in [J]}
	 T_j^{*}T_j \bigr) P_m
	\right\|\leq \left\|
	\sum_{j\in [J]}
	 T_j^{*}T_j 
	\right\|\leq \left\|
	\sum_{j\in \mathbb{Z}}
	T_j^{*}T_j 
	\right\|,
	\]
it follows that $\left\{
\widetilde{p}_{(J,m)}
\right\}_{J,m\in\mathbb{N}}.
$ satisfies the condition $(L).$
		
		 Let now $\mathcal{U}$ be the unitary shift operator on $H$ given by $\mathcal{U}(e_j) = e_{j+1} $ for all $j \in \mathbb{Z} .$ We define the map $\Phi : \mathcal{A}_1 \to \mathcal{A}_1$ by
		$$
		\Phi \left( \{ y_j \}_{j \in \mathbb{Z}} \right) 
		= \{ \mathcal{U}^{*} y_{j - 1}  \mathcal{U} \}_{j \in \mathbb{Z}},
		$$ for all $ \{ y_j \}_{j \in \mathbb{Z}} \in \mathcal{A}_1 ,$ 
		(that is the $j$-th coordinate of $\Phi (\{y_j\}_{j\in\mathbb{Z}}) = \mathcal{U}^{*} y_{j - 1}  \mathcal{U}$ for all $j \in \mathbb{Z}$). It is straightforward to check that $\Phi$ is an isometric algebra isomorphism of $\mathcal{A}_1 $ (in the norm $ \| \cdot \|_1 $) such that the restriction of $\Phi$ to $\mathcal{A} $ is an isometric isomorphism of $\mathcal{A} $ in the norm $ \| \cdot \| .$ It remains to show that $\Phi$ is strongly aperiodic with respect to
		$\left\{
		\widetilde{p}_{(J,m)}
		\right\}_{J,m\in\mathbb{N}}.
		$ To prove this, we need first the following auxiliary technical lemma.
		\begin{lemma}\label{pomocno}
			Let $\{ \tilde{P_j}\}_{j\in\mathbb{N}}$ be a sequence of mutually
			orthogonal projections in $B(H)$, and $F_j\in B(H)$
			for all $j\in\mathbb{N}$. Then the series 
			\[
			\sum_{j=1}^{\infty}
			 \tilde{P_j}F_j \tilde{P_j}
			\]
			is convergent in $B(H)$ if and only if $\lim_{j\to\infty}\left\|
			\tilde{P_j}F_j \tilde{P_j}
			\right\|=0, $ and in this case
			\[
			\left\|
			\sum_{j=1}^{\infty}
			 \tilde{P_j}F_j \tilde{P_j}
			\right\|
			=
			\sup_{j\in\mathbb{N}}
			\left\|
			 \tilde{P_j}F_j \tilde{P_j}
			\right\|.
			\]
		\end{lemma}
		
		\begin{proof}
			Suppose first that the series 
			$
			\sum_{j=1}^{\infty}
			 \tilde{P_j}F_j \tilde{P_j}
			$
			is convergent in $B(H).$ Then, given $ \epsilon >0, $ there exists some
			$M\in\mathbb{N}$ such that
			\[
			\left\|
			\sum_{n=m_{1}}^{m_{1}}
			 \tilde{P_n}F_n \tilde{P_n}
			\right\|< \epsilon
			\]
			for all $m_{1} \in\mathbb{N}$ with
			$M\leq m_{1}\leq m_{2}$. Thus,
			\[
			\| \tilde{P_n}F_n \tilde{P_n}\|<\epsilon
			\]
			for all $n\geq M$, so $
		\lim_{n \to \infty}	\| \tilde{P_n}F_n \tilde{P_n}\|=0.$
			Since $P_{j-s}^{\prime}$ are mutually orthogonal, for every
			$h\in H$ and  each $i, n\in\mathbb{N}$ we have
			\[
			\left\|
			( \sum_{j=1}^{\infty}
			 \tilde{P_j}F_j \tilde{P_j} )h
			\right\|^{2}
			=
			\left\|
			\sum_{j=1}^{\infty}
			( \tilde{P_j}F_j \tilde{P_j}h)
			\right\|^{2}
			=
			\sum_{j=1}^{\infty}
			\left\|
			 \tilde{P_j}F_j \tilde{P_j}h
			\right\|^{2}
			\]
			\[
			\geq
			\left\|
			 \tilde{P_i}F_i \tilde{P_i}h
			\right\|^{2}
			.\]
			Taking supremum over all $h\in H$ with $\|h\|=1$, we obtain
			\[
			\left\|
			\sum_{j=1}^{\infty}
			 \tilde{P_j}F_j \tilde{P_j}
			\right\|
			\geq
			\| \tilde{P_i}F_i \tilde{P_i}\|
			\]
			for each $i\in\mathbb{N},$ hence \[
			\left\|
			\sum_{j=1}^{\infty}
			 \tilde{P_j}F_j \tilde{P_j}
			\right\|
			\geq
			\sup_{j\in\mathbb{N}}
			\| \tilde{P_j}F_j \tilde{P_j}\|.
			\]
			On the other hand, since 
			\[ 	\left\|
			( \sum_{j=1}^{\infty}
			\tilde{P_j}F_j \tilde{P_j} )h
			\right\|^{2}
			=
			\sum_{j=1}^{\infty}
			\| \tilde{P_j}F_j \tilde{P_j}h\|^{2}= \sum_{j=1}^{\infty}
			\| \tilde{P_j}F_j \tilde{P_j}^2 h\|^{2}
			\leq
			\sum_{j=1}^{\infty}
			\| \tilde{P_j}F_j \tilde{P_j}\|^{2}\,
			\| \tilde{P_j}h\|^{2}
			\]
			\[
			\leq
			\left( 
			\sup_{j\in\mathbb{N}}
			\| \tilde{P_j}F_j \tilde{P_j}\|^{2}
			\right)
			\sum_{j=1}^{\infty}
			\| \tilde{P_j}h\|^{2}	\leq
			\left(
			\sup_{j\in\mathbb{N}}
			\| \tilde{P_j}F_j \tilde{P_j}\|
			\right)^{2}
			\|h\|^{2},
			\]
			we deduce that
			\[
			\left\|
			\sum_{j=1}^{\infty}
			 \tilde{P_j}F_j \tilde{P_j}
			\right\|
			\leq
			\sup_{j\in\mathbb{N}}
			\| \tilde{P_j}F_j \tilde{P_j}\|.
			\] Thus, in this case,	\[
			\left\|
			\sum_{j=1}^{\infty}
			\tilde{P_j}F_j \tilde{P_j}
			\right\|
			=
			\sup_{j\in\mathbb{N}}
			\| \tilde{P_j}F_j \tilde{P_j}\|.
			\] 
			
			Suppose now that $\lim_{j\to\infty}\left\|
			\tilde{P_j}F_j \tilde{P_j}
			\right\|=0. $ Then, given $ \epsilon >0, $ there exists some
			$M\in\mathbb{N}$ such that $\left\|
			\tilde{P_j}F_j \tilde{P_j}
			\right\| < \epsilon $ for all $ j \geq M.$ By similar calculations as above, if $n_1, n_2 \in \mathbb N $ with $ n_2>n_1>M ,$ we deduce that for every $ h \in H$ that $$	\left\|
			( \sum_{j=n_1}^{n_2}
			\tilde{P_j}F_j \tilde{P_j} )h
			\right\|^{2}
			\leq \left(
			\sup_{n_1 \leq j \leq n_2}
			\| \tilde{P_j}F_j \tilde{P_j}\|
			\right)^{2}
			\|h\|^{2}  ,$$ so \[
			\left\|
			\sum_{j=n_1}^{n_2}
			\tilde{P_j}F_j \tilde{P_j}
			\right\|
			\leq
			\sup_{n_1 \leq j \leq n_2}
			\| \tilde{P_j}F_j \tilde{P_j}\| < \epsilon.
			\] This holds for all $n_1, n_2 \in \mathbb N $ with $ n_2>n_1>M .$ Since $ \epsilon >0 $ was chosen arbitrarily, we deduce that the series $
			\sum_{j=1}^{\infty}
			\tilde{P_j}F_j \tilde{P_j}
			$
			is convergent in $B(H).$ \end{proof}

		\begin{remark}
			Lemma \ref{pomocno} remains valid if we instead of the convergence
			in operator norm of
			$\sum_{n=1}^{\infty}
			 \tilde{P_n}F_n \tilde{P_n},
			$
			just assume that
			$\{\| \tilde{P_n}F_n \tilde{P_n}\|\}_{n\in\mathbb{N}}$
			is bounded. In this case, the sum
			\[
			\sum_{n=1}^{\infty}
			 \tilde{P_n}F_n  \tilde{P_n}
			\]
			is strongly convergent and the same proof as above applies.
		\end{remark}
			\begin{lemma}\label{aperiodic}
			The map $\Phi$ is strongly aperiodic with respect to
			$\left\{
			\widetilde{p}_{(J,m)}
			\right\}_{J,m\in\mathbb{N}}.
			$ 
		\end{lemma}
			\begin{proof}
			Let $a=\{F_j\}_{j\in\mathbb{Z}}\in \mathcal{A}$ and
			$J,m \in\mathbb{N}$ be given. Choose any
			$d\in\mathbb{N}$ with $d>2(m+J)$.
			Given $\varepsilon>0$, we can choose some $M\in\mathbb{N}$ such that
			\[
			\left\|
			\sum_{j=M}^{\infty}
			F_j^{*}F_j
			\right\|
			<
			\varepsilon.
			\]
			For every $m_2>m_1>M$, we have
			\[
			\left\|
			\sum_{n=m_1}^{m_2}
			a\Phi^{nd}
			\bigl(
			\widetilde{p}_{(J,m)}
			\bigr)
			\right\|
			=
			\]
			\[
			=
			\left\|
			\sum_{j\in[J]}
			\sum_{n=m_1}^{m_2}
			\mathcal{U}^{*nd}P_m \mathcal{U}^{nd}
			F_{j+nd}^{*}F_{j+nd}
			\mathcal{U}^{*nd}P_m \mathcal{U}^{nd}
			\right\|
			\]
			\[
			\leq
			\left\|
			\sum_{j\in[J]}
			\sum_{n=m_1}^{m_2}
			\mathcal{U}^{*nd}P_m
			\|F_{j+nd}\|^{2}
			P_m \mathcal{U}^{nd}
			\right\|
			\]
			\[
			\leq
			\left\|
			\sum_{j\in[J]}
			\sum_{n=m_1}^{m_2}
			\mathcal{U}^{*nd}P_m
			\left\|
			\sum_{j=M}^{\infty}
			F_j^{*}F_j
			\right\|
			P_m \mathcal{U}^{nd}
			\right\|
			\leq
			\]
			\[
			\leq
			\varepsilon
			\left\|
			\sum_{j\in[J]}
			\sum_{n=m_1}^{m_2}
			\mathcal{U}^{*nd}P_m \mathcal{U}^{nd}
			\right\|\leq
			2J\varepsilon
			\left\|
			\sum_{n=m_1}^{m_2}
			\mathcal{U}^{*nd}P_m \mathcal{U}^{nd}
			\right\|,
			\]
			where in the second inequality we have used that
			$d>2J$, hence
			\[
			j+nd>-J+J(M+1)>M
			\] for all $m_1,n, m_2 \in \mathbb{N} $ with $m_2\geq n \geq m_1>M+1$. Now, since $d>2m$, we get
			\[
			\{-m-n_{1}d,\ldots,m-n_{1}d\}
			\cap
			\{-m-n_{2}d,\ldots,m-n_{2}d\}
			=
			\emptyset
			\]
			whenever $n_{1},n_{2}\in\mathbb{N}$ with
			$n_{1}\neq n_{2}$. For each $n\in\mathbb{N}$ let $Q_{n}$ be the orthogonal
			projection onto
			\[
			\operatorname{Span}
			\{e_{j}\}_{j=-m-nd}^{\,m-nd}.
			\]
			Then $\{Q_{n}\}_{n\in\mathbb{N}}$ is a sequence of mutually
			orthogonal projections, and
			\[
			\mathcal{U}^{*nd}P_{m}\mathcal{U}^{nd}
			=
			Q_{n}
			\] for all $n\in\mathbb{N}$. Therefore,
			\[ \sum_{n=m_{1}}^{m_{2}}
			\mathcal{U}^{*nd}P_{m}\mathcal{U}^{nd}
			\leq
			I,
			\] so we deduce that
			
			\[
			\left\|
			\sum_{n=m_{1}}^{m_{2}}
			a\Phi^{nd}
			\bigl(
			\widetilde{p}_{(J,m)}
			\bigr)
			\right\|
			\leq
			2J\varepsilon
			\]
			
			for all $m_{2}>m_{1}>M+1$. Thus, the sum
			$\sum_{n=1}^{\infty}
			a\Phi^{nd}
			\bigl(
			\widetilde{p}_{(J,m)}
			\bigr)$
			is convergent in $\mathcal{A}$ since $ \epsilon >0 $ was chosen arbitrarily. Moreover, since
			$
			j+ndk-J+J(M+1)>M
			$
		 for all $j\in[J], n\in\mathbb{N}$ and each $ k \in \mathbb N $ with $ k > M+1 ,$
			by the same calculations as above we can obtain that
			\[
			\left\|
			\sum_{n=1}^{r}
			a\Phi^{ndk}
			\bigl(
			\widetilde{p}_{(J,m)}
			\bigr)
			\right\| \leq 	2J\varepsilon
			\] for every $ r \in \mathbb N $ and each $ k \in \mathbb N $ with $ k > M+1 .$ Hence \[
			\left\|
			\sum_{n=1}^{\infty}
			a\Phi^{ndk}
			\bigl(
			\widetilde{p}_{(J,m)}
			\bigr)
			\right\| \leq 	2J\varepsilon
			\] for each $ k \in \mathbb N $ with $ k > M+1 ,$ from which we can conclude that 
			\[
			\lim_{k\to\infty}
			\left\|
			\sum_{n=1}^{\infty}
			a\Phi^{ndk}
			\bigl(
			\widetilde{p}_{(J,m)}
			\bigr)
			\right\|
			=
			0.
			\] again since $ \epsilon >0 $ was chosen arbitrarily.
			Similarly, by using that
			\[
			\lim_{M\to\infty}
			\sum_{j=M}^{\infty}
			F_{-j}^{*}F_{-j}
			=
			0,
			\]
			we can show that the sum
			\[
			\sum_{n=1}^{\infty}
			a\,
			\Phi^{-nd}
			\Bigl(
			\widetilde{p}_{(J,m)}
			\Bigr)
			\]
			is convergent and
			\[
			\lim_{k\to\infty}
			\sum_{n=1}^{\infty}
			a\,
			\Phi^{-ndk}
			\Bigl(
			\widetilde{p}_{(J,m)}
			\Bigr)
			=
			0,
			\]	where the corresponding series is convergent for each $ k \in \mathbb{N}.$	 
			\end{proof} 
		Let $\{ W_{j} \}_{j \in \mathbb{Z}}$ be
		a sequence of operators in $B(H)$ which is uniformly bounded in norm such that each $W_{j}$ has a bounded inverse and such that $\{ W_{j}^{-1} \}_{j \in \mathbb{Z}}$ is also uniformly bounded in the norm. Then let $ b = \{ W_{j}\mathcal{U} \}_{j \in \mathbb{Z}} .$ Obviously, $ b $   is invertible in $\mathcal{A}_1.$\\
		Now we are ready to state the second of the two main theorems in this paper.  
	\begin{theorem}\label{period-kompakt}
		Under the above notation and assumptions the following statements are equivalent.
		
		(1) The set of periodic elements of $T$ is dense in $\mathcal{A}.$
		
		(2) There exist a dense subset $  \mathcal{D}$ of $\mathcal{A}$ such that for each $ J, m \in \mathbb{N} $ and each $\{D_j\}_{j\in\mathbb Z} \in \mathcal{D}$ there exists some  $ d>2m $ satisfying that $$\lim_{ n\to\infty}	
		\left\|
		W_{r+ nd}
		\dots
		W_{r+1}
		D_r
		P_m
		\right\|
		=0$$ and \[
		\lim_{n\to\infty}
		\left\|
		W_{r+1-nd}^{-1}
		\dots
		W_{r}^{-1}
		D_rP_m
		\right\|
		=
		0
		\]for each $r \in [J].$
		
		(3) For each $r\in\mathbb Z,$ there exists a dense subset $\mathcal D_H^r$ of $H$ such that for every $x\in \mathcal D_H^r$ we can find some $ d \in \mathbb N$ satisfying
		\[
		\lim_{n\to\infty}\|W_{r+nd}\cdots W_{r+1}x\|
		=\lim_{n\to\infty}\|W_{r+1-nd}^{-1}\cdots W_r^{-1}x\|=0.
		\]
		
		(4) The operator $T$ is Devaney chaotic on $ \mathcal A.$
	\end{theorem}
	
	\begin{proof}
	$(1)\Rightarrow(2)$ 
		If (1) holds, then by Theorem \ref{glavno-haos} there is a dense subset $\mathcal{D} \subseteq \mathcal{A}$ such that for each $\{D_j\}_{j\in\mathbb Z} \in \mathcal{D}$ and every $ J, m \in \mathbb{N} $ we can find some  $ d>2m $ such that the infinite sum
		\[
		\sum_{n=1}^{\infty}
		T^{nd}
		\bigl(
		\{D_j\}_{j\in\mathbb Z}
		\bigr)
		\,
		\Phi^{nd}
		\Bigl(
		\widetilde{p}_{(J,m)}
		\Bigr)
		\]
		is convergent. Noticing that
		\[
		\Bigl(
		T^{nd}
		\bigl(
		\{D_j\}_{j\in\mathbb Z}
		\bigr)
		\,
		\Phi^{nd}
		\Bigl(
		\widetilde{p}_{(J,m)}
		\Bigr)
		\Bigr)_{r+nd}
		\]
		\[
		=
		W_{r+nd}\dots W_{r+1}\,
		D_r\,
		P_m\,\mathcal{U}^{nd}
		\]
		if $r\in[J]$, and $0$ else, we obtain
		for each $\widetilde{m}\in\mathbb N$ and $r\in[J]$ that
		\[
		\left\|
		\sum_{n=\widetilde{m}}^{\infty}
		T^{nd}
		\bigl(
		\{D_j\}_{j\in\mathbb Z}
		\bigr)
		\,
		\Phi^{nd}
		\Bigl(
		\widetilde{p}_{(J,m)}
		\Bigr)
		\right\|^{2}
		\]
		\[
		\geq
		\left\|
		\sum_{n=\widetilde{m}}^{\infty}
		\mathcal{U}^{*nd}
		P_m
		D_r^{*}
		W_{r+1}^{*}
		\dots
		W_{r+nd}^{*}
		W_{r+nd}
		\dots
		W_{r+1}
		D_r
		P_m
		\mathcal{U}^{nd}
		\right\|.
		\]
		Since for each $ n \in \mathbb{N}$ we have
		\[
		P_m \mathcal{U}^{nd}
		=
		P_m \mathcal{U}^{nd}Q_n,
		\]
		where each $Q_n$ is as in the proof of Lemma \ref{aperiodic} and $Q_n^{\prime}$-s are mutually
		orthogonal, by Lemma \ref{pomocno} we get
		\[
		\left\|
		\sum_{n=\widetilde{m}}^{\infty}
		\mathcal{U}^{*nd}
		P_m
		D_r^{*}
		W_{r+1}^{*}
		\dots
		W_{r+nd}^{*}
		W_{r+nd}
		\dots
		W_{r+1}
		D_r
		P_m
		\mathcal{U}^{nd}
		\right\|
		\]
		\[
		=
		\sup_{\substack{n\in\mathbb N\\ n\geq \widetilde{m}}}
		\left\{
		\left\|
		\mathcal{U}^{*nd}
		P_m
		D_r^{*}
		W_{r+1}^{*}
		\dots
		W_{r+nd}^{*}
		W_{r+nd}
		\dots
		W_{r+1}
		D_r
		P_m
		\mathcal{U}^{nd}
		\right\|
		\right\}
		\]
		\[
		=
		\sup_{\substack{n\in\mathbb N\\ n\geq \widetilde{m}}}
		\left\{
		\left\|
		W_{r+nd}
		\dots
		W_{r+1}
		D_r
		P_m
		\right\|^{2}
		\right\}
		\]
		for each $\tilde{m} \in \mathbb{N}$ and $r \in [J].$
		Hence,
		\[
		0
		=
		\lim_{\widetilde{m}\to\infty}
		\left\|
		\sum_{n=\widetilde{m}}^{\infty}
		T^{nd}
		\bigl(
		\{D_j\}_{j\in\mathbb Z}
		\bigr)
		\,
		\Phi^{nd}
		\Bigl(
		\widetilde{p}_{(J,m)}
		\Bigr)
		\right\|
		\geq \lim_{\widetilde{m}\to\infty}	\sup_{\substack{n\in\mathbb N\\ n\geq \widetilde{m}}}
		\left\{
		\left\|
		W_{r+nd}
		\dots
		W_{r+1}
		D_r
		P_m
		\right\|^{2}
		\right\}
		\]
		for each $r\in[J]$. Therefore, for each $r\in[J]$ we have  $$\lim_{ n\to\infty}	
		\left\|
		W_{r+ nd}
		\dots
		W_{r+1}
		D_r
		P_m
		\right\|=\lim_{\widetilde{m}\to\infty}	\sup_{\substack{n\in\mathbb N\\ n\geq \widetilde{m}}}
		\left\{
		\left\|
		W_{r+nd}
		\dots
		W_{r+1}
		D_r
		P_m
		\right\|
		\right\}
		=0$$ In fact, strictly speaking,
		we should show first that the infinite sum
		\[
		\sum_{n=\widetilde{m}}^{\infty}
		\mathcal{U}^{*nd}
		P_m
		D_r^{*}
		W_{r+1}^{*}
		\dots
		W_{r+nd}^{*}
		W_{r+nd}
		\dots
		W_{r+1}
		D_r
		P_m
		\mathcal{U}^{nd}
		\]
		is convergent in $\mathcal{C}$ for each
		$\widetilde{m}\in\mathbb{N}$.
		To this end, let $\varepsilon>0$ be given and choose
		some $M\geq \widetilde{m}$ such that
		
		\[
		\left\|
		\sum_{n=\widetilde{m}_{1}}^{\widetilde{m}_{2}}
		T^{nd}
		\bigl(
		\{D_j\}_{j\in\mathbb Z}
		\bigr)
		\,
		\Phi^{nd}
		\Bigl(
		\widetilde{p}_{(J,m)}
		\Bigr)
		\right\|
		<
		\varepsilon
		\]
		
		for all
		$\widetilde{m}_{1},\widetilde{m}_{2}\in\mathbb{N}$
		with
		$M\leq \widetilde{m}_{1}<\widetilde{m}_{2}.$
		Now observe that
		\[
		\left\|
		\sum_{n=\widetilde{m}_{1}}^{\widetilde{m}_{2}}
		T^{nd}
		\bigl(
		\{D_j\}_{j\in\mathbb Z}
		\bigr)
		\,
		\Phi^{nd}
		\Bigl(
		\widetilde{p}_{(J,m)}
		\Bigr)
		\right\|^{2}
		=
		\]		
		\[
		=
		\left\|
		\sum_{n=\widetilde{m}_{1}}^{\widetilde{m}_{2}}
		\sum_{i\in[J]}
		\mathcal{U}^{*nd}
		P_m
		D_i^{*}
		W_{i+1}^{*}
		\dots
		W_{i+nd}^{*}
		W_{i+nd}
		\dots
		W_{i+1}
		D_i
		P_m
		\mathcal{U}^{nd}
		\right\|
		\]
		\[
		=
		\left\|
		\sum_{i\in[J]}
		\sum_{n=\widetilde{m}_{1}}^{\widetilde{m}_{2}}
		\mathcal{U}^{*nd}
		P_m
		D_i^{*}
		W_{i+1}^{*}
		\dots
		W_{i+nd}^{*}
		W_{i+nd}
		\dots
		W_{i+1}
		D_i
		P_m
		\mathcal{U}^{nd}
		\right\|
		\]
		\[
		\geq
		\left\|
		\sum_{n=\widetilde{m}_{1}}^{\widetilde{m}_{2}}
		\mathcal{U}^{*nd}
		P_m
		D_r^{*}
		W_{r+1}^{*}
		\dots
		W_{r+nd}^{*}
		W_{r+nd}
		\dots
		W_{r+1}
		D_r
		P_m
		\mathcal{U}^{nd}
		\right\|
		\]
		for each $r\in[J]$. Hence, from the convergence of the sum \[
		\sum_{n=\widetilde{m}}^{\infty}
		T^{nd}
		\bigl(
		\{D_j\}_{j\in\mathbb Z}
		\bigr)
		\,
		\Phi^{nd}
		\Bigl(
		\widetilde{p}_{(J,m)}
		\Bigr),
		\] we can deduce the convergence of the sum 
		\[
		\sum_{n=\widetilde{m}}^{\infty}
		\mathcal{U}^{*nd}
		P_m
		D_r^{*}
		W_{r+1}^{*}
		\dots
		W_{r+nd}^{*}
		W_{r+nd}
		\dots
		W_{r+1}
		D_r
		P_m
		\mathcal{U}^{nd}
		\] for each $r\in[J]$ and $\widetilde{m}\in\mathbb{N}$.
		
		Similarly, from the convergence of the sum
		\[
		\sum_{n=1}^{\infty}
		S^{nd}
		\bigl(
		\{D_j\}_{j\in\mathbb Z}
		\bigr)
		\,
		\Phi^{-ind}
		\Bigl(
		\widetilde{p}_{(J,m)}
		\Bigr),
		\]
		we obtain that
		 \[
		\lim_{n\to\infty}
		\left\|
		W_{r+1-nd}^{-1}
		\dots
		W_{r}^{-1}
		D_rP_m
		\right\|
		=
		0
		\]
		for each $r\in[J].$ 
		
		$(2)\Rightarrow(1)$
		Let $\mathcal{D}$ be a dense subset of $\mathcal{A}$
		satisfying the assumptions in 2).
	 Suppose that $\{D_j\}_{j\in\mathbb Z}\in\mathcal D$ and $ J, m \in \mathbb{N} $ are given. Choose $d\in\mathbb N$
		satisfying the assumptions of (2) with respect to $J,m$
		and
		$
		\{D_{-J},\ldots,D_{J}\}
		.$ Then, 
		for every $\varepsilon>0$ we can find some
		$M\in\mathbb N$ such that
		\[
		\sup_{\widetilde{m}_{1}\leq n\leq \widetilde{m}_{2}}
		\left\{
		\left\|
		W_{r+nd}\dots W_{r+1}D_rP_m
		\right\|
		\right\}
		<
		\varepsilon
		\]
		
		for all
		$
		\widetilde{m}_{2}>\widetilde{m}_{1}>M
		$
		and each $r\in[J]$.
		 Further, since $ d>2m ,$ we have
		\[
		\left\|
		\sum_{n=\widetilde{m}_{1}}^{\widetilde{m}_{2}}
		T^{nd}
		\bigl(
		\{D_j\}_{j\in\mathbb Z}
		\bigr)
		\,
		\Phi^{nd}
		\Bigl(
		\widetilde{p}_{( J,m)}
		\Bigr)
		\right\|^{2}
		=
		\]
		\[
		=
		\left\|
		\sum_{r\in[J]}
		\sum_{n=\widetilde{m}_{1}}^{\widetilde{m}_{2}}
		\mathcal{U}^{*nd}
		P_m
		D_r^{*}
		W_{r+1}^{*}
		\dots
		W_{r+nd}^{*}
		W_{r+nd}
		\dots
		W_{r+1}
		D_r
		P_m
		\mathcal{U}^{nd}
		\right\|
		\]
		\[
		\leq
		\sum_{r\in[J]}
		\left\|
		\sum_{n=\widetilde{m}_{1}}^{\widetilde{m}_{2}}
		\mathcal{U}^{*nd}
		P_m
		D_r^{*}
		W_{r+1}^{*}
		\dots
		W_{r+nd}^{*}
		W_{r+nd}
		\dots
		W_{r+1}
		D_r
		P_m
		\mathcal{U}^{nd}
		\right\|
		\]
		\[
		=
		\sum_{r\in[J]}
		\left(
		\sup_{\widetilde{m}_{1}\leq n\leq \widetilde{m}_{2}}
		\left\{
		\left\|
		W_{r+nd}
		\dots
		W_{r+1}
		D_rP_m
		\right\|^{2}
		\right\}
		\right)
		\leq
		2J\varepsilon
		\]
		for all $\widetilde{m}_{1},\widetilde{m}_{2}\in\mathbb N$
		with
		$
		M\leq \widetilde{m}_{1}<\widetilde{m}_{2}.
		$
		Hence the sum
		\[
		\sum_{n=1}^{\infty}
		T^{nd}
		\bigl(
		\{D_j\}_{j\in\mathbb Z}
		\bigr)
		\,
		\Phi^{nd}
		\Bigl(
		\widetilde{p}_{(J,m)}
		\Bigr)
		\]
		
		is convergent in $\mathcal A$.
		Moreover, since obviously
		\[
		\sup_{n\in\mathbb N}
		\left\{
		\left\|
		W_{r+n kd}
		\dots
		W_{r+1}
		D_rP_m
		\right\|^{2}
		\right\}
		\leq
		\]	
		\[
		\leq
		\sup_{\substack{n\in\mathbb N\\ n\geq k}}
		\left\{
		\left\|
		W_{r+nd}
		\dots
		W_{r+1}
		D_rP_m
		\right\|
		\right\}
		\]
		for each $k \in\mathbb N$ and
		$r\in[J]$,
		by similar calculations as above for each $k \in \mathbb{N}$ we obtain
		\[
		\left\|
		\sum_{n=1}^{\infty}
		T^{ndk}
		\bigl(
		\{D_j\}_{j\in\mathbb Z}
		\bigr)
		\,
		\Phi^{ndk}
		\Bigl(
		\widetilde{p}_{(J,m)}
		\Bigr)
		\right\|^{2}
		\]
		
		\[
		\leq
		\sum_{r\in[J]}
		\left(
		\sup_{\substack{n\in\mathbb N}}
		\left\{
		\left\|
		W_{r+n kd}
		\dots
		W_{r+1}
		D_rP_m
		\right\|^{2}
		\right\}
		\right)
		\longrightarrow 0,
		\qquad
		k\to\infty .
		\] Similarly, from
		\[ \lim_{n\to\infty}
		\left\|
		W_{r+1-nd}^{-1}
		\dots
		W_{r}^{-1}
		D_rP_m
		\right\|=
		\lim_{ k\to\infty}
		\left(
		\sup_{\substack{n\in\mathbb N\\ n\geq k}}
		\left\|
		W_{r-nd+1}^{-1}
		\dots
		W_{r}^{-1}
		D_rP_m
		\right\|
		\right)
		=
		0
		\]
		for each $r\in[J]$, we deduce that the sum
		\[
		\sum_{n=1}^{\infty}
		S^{ndk}
		\bigl(
		\{D_j\}_{j\in\mathbb Z}
		\bigr)
		\,
		\Phi^{-ndk}
		\Bigl(
		\widetilde{p}_{(J,m)}
		\Bigr)
		\]
		is convergent for each $ k \in \mathbb{N} $, and
		\[
		\lim_{ k\to\infty}
		\sum_{n=1}^{\infty}
		S^{nd k}
		\bigl(
		\{D_j\}_{j\in\mathbb Z}
		\bigr)
		\,
		\Phi^{-ndk}
		\Bigl(
		\widetilde{p}_{(J,m)}
		\Bigr)
		=
		0.
		\]
		By Theorem \ref{glavno-haos} it follows that the set of periodic elements of $T$ is dense in $\mathcal{A} .$
		
		$(2)\Rightarrow(3)$ Suppose that (2) holds and let $ \mathcal D$ be a dense subset of $ \mathcal A$ which satisfies the assumptions in (2). For every $ r \in \mathbb Z $ and $ x = \lbrace x_j \rbrace_{ j \in \mathbb Z}\in \mathcal D ,$ let $ x_r$ denote the $r$-th component of $x,$ and set $$\tilde {\mathcal D}^r:= \lbrace  x_r \text{ } \vert \text{ } x \in \mathcal D\rbrace .$$ It is clear that $ \tilde {\mathcal D}^r$ is dense in $ \mathcal C$ for every $ r \in \mathbb Z .$ Put then $$ \mathcal D_H^r := \lbrace  D_r P_m h \text{ } \vert \text{ } h \in H, \text{ } m \in \mathbb N, \text{ } D_r \in \tilde {\mathcal D}^r \rbrace  .$$ Then, $\mathcal D_H^r $ is dense in $H$ for every $r \in \mathbb Z .$ To see this, let $ h \in H $ with $ h \neq 0$ and $ \epsilon >0 $ be given. We can find some $ m \in \mathbb N $ such that $$ \|  h - P_m h  \| < \frac{\epsilon}{2 } .
		$$ Since $\tilde {\mathcal D}^r $ is dense in $ \mathcal C $  for every $r \in \mathbb Z $, we can find some $ D_r \in \tilde {\mathcal D}^r $ such that $$\|  D_r - P_m   \| < \frac{\epsilon}{2\| h \| } . $$
		Hence,  $$ \|  (D_r P_m - P_m) h  \|= \|  (D_r - P_m) P_mh  \|< \frac{\epsilon}{2 } ,
		$$ so $$ \|  D_r P_m  h  -h \| < \epsilon,
		$$ which shows the density of $\mathcal D_H^r $ in $H$ for every $r \in \mathbb Z .$
		
		 Now given $ r \in \mathbb Z$ and  $ x \in \mathcal D_H^r,$ let $m \in \mathbb N $ and $\{D_j\}_{j\in\mathbb Z} \in \mathcal{D}$ be such that $ x = D_r P_m h$ for some $h \in H. $ Choose any $J \in \mathbb N$ with $ J > \vert r \vert .$ Then we can find some $ d >2m $ satisfying the assumptions in (2) with respect to  $J,m$ and $\{D_j\}_{j\in\mathbb Z} .$ It follows then
		  that $d$ has the desired properties of (3).
		
$(3)\Rightarrow(2)$ If (3) holds, for each $ r \in \mathbb Z$ let $\widetilde{ \mathcal D}^r $
consist of all operators $D_r$ in $\mathcal C$ with the property that $D_r(e_j)\in \mathcal D_H^r \cup \lbrace 0 \rbrace $ for every $j\in\mathbb Z$. Then  $\widetilde{ \mathcal D}^r $ is dense in $\mathcal C$ for each $ r \in \mathbb Z.$ Indeed, given $F\in \mathcal C$ and $\varepsilon>0$, we can find some $m\in\mathbb N$ such that $\|F-FP_m\|<\frac{\varepsilon}{2}$. For each $j\in\{-m,\ldots,m\}$ there exists some $d_j\in \mathcal D_H^r$ such that
\[
\|d_j-Fe_j\|<\frac{\varepsilon}{2(2m+1)}.
\]
Let $D_r$ be the finite rank operator on $H$ given by
\[
D_r x=\sum_{j=-m}^{m}\langle x,e_j\rangle d_j\quad\text{for all }x\in H.
\]
Then $D_r \in  \widetilde{ \mathcal D}^r$. Moreover, since $D_r P_m=D_r $, we get
\begin{align*}
	\|(D_r-FP_m)(x)\|&=\left\|\sum_{j=-m}^{m}\langle x,e_j\rangle(d_j-Fe_j)\right\|\\
	&\leq\sum_{j=-m}^{m}|\langle x,e_j\rangle|\,\|d_j-Fe_j\|\\
	&\leq(2m+1)\|x\|\frac{\varepsilon}{2(2m+1)}=\|x\|\frac{\varepsilon}{2}
\end{align*}
for all $x\in H$, so $\|D_r -FP_m\|<\frac{\varepsilon}{2}$. Thus, $\|D_r-F\|<\varepsilon$, which shows that  $\widetilde{ \mathcal D}^r$ is dense in $\mathcal C.$ This holds for every $ r \in \mathbb Z.$ Then, since the set of all sequences in $\mathcal A $ with finitely many non-zero components is dense in $\mathcal A $, it follows that the set
\[
\mathcal D:=\{\{D_j\}_{j\in\mathbb Z}\in\mathcal A\mid D_j\in  \widetilde{ \mathcal D}^j \text{ for all }j\}
\]
is dense in $\mathcal A$. Suppose now that $J,m\in\mathbb N$ and $\{D_j\}_{j\in\mathbb Z}\in \mathcal D$ are given. Let $r\in[J]$. Then $D_r(e_j)\in \mathcal D_H^r \cup \lbrace 0 \rbrace$ for every $j\in\{-m,\ldots,m\}$. 
By the assumption, given $\varepsilon>0$, for each  $j\in\{-m,\ldots,m\}$ we can find some $N_j, d_j  \in\mathbb N$ such that for each $n\geq N_j$  it holds that
\[
\|W_{r+nd_j}\cdots W_{r+1}D_r(e_j)\|<\frac{\varepsilon}{2m+1}
\]
and
\[
\|W_{r+1-nd_j}^{-1}\cdots W_r^{-1}D_r(e_j)\|<\frac{\varepsilon}{2m+1}.
\]
Let $ d^{(r)}:= 2 m d_{-m} d_{-m+1} \cdots d_{m-1} d _m,$ then $d^{(r)} > d_j $ for all  $j\in\{-m,\ldots,m\}.$ Moreover, $d^{(r)}> 2m. $ Set also $ N^{(r)}:= N_{-m} N_{-m+1} \cdots N_{m-1} N _m,.$ Then, for all $j\in\{-m,\ldots,m\}$ we get \[
\|W_{r+n d^{(r)}}\cdots W_{r+1}D_r(e_j)\|<\frac{\varepsilon}{2m+1}
\]
and
\[
\|W_{r+1-n d^{(r)}}^{-1}\cdots W_r^{-1}D_r(e_j)\|<\frac{\varepsilon}{2m+1}.
\] for each $ n > N^{(r)} .$ This can be done for every $r\in[J]$.
By the same estimate as above, we obtain then
\[
\|W_{r+nd^{(r)}}\cdots W_{r+1}D_rP_m\|<\varepsilon \text{ and}\qquad
\|W_{r+1-nd^{(r)}}^{-1}\cdots W_r^{-1}D_rP_m\|<\varepsilon
\]
for all $n\geq  N^{(r)} $. Set now $ d= d^{(-J)}d^{(-J+1)} \cdots d^{(J-1)}d^{(J)}.$ Then $ d >2m $ and \[
\|W_{r+nd}\cdots W_{r+1}D_rP_m\|<\varepsilon \text{ and}\qquad
\|W_{r+1-nd}^{-1}\cdots W_r^{-1}D_rP_m\|<\varepsilon
\] for each $ n > N^{(r)} $ and every $r\in[J].$ Since $\varepsilon>0$ was chosen arbitrarily, we deduce 
\[
\lim_{n\to\infty}\|W_{r+nd}\cdots W_{r+1}D_rP_m\|
=\lim_{n\to\infty}\|W_{r+1-nd}^{-1}\cdots W_r^{-1}D_rP_m\|=0
\]
for every $r\in[J]$, so the conditions of (2) are satisfied.

$(2)\Rightarrow(4)$ If (2) holds, then the set of periodic elements of $T$ is dense in $ \mathcal A$ by the implication $(2)\Rightarrow(1)$ which we already proved above. Therefore, it suffices to show that $T$ is topologically transitive on $ \mathcal A.$ To this end, let $ \mathcal D$ be the dense set from the assumption in (2). Since $\mathcal D$ is dense in $\mathcal A$, given $J,m\in\mathbb N$ we can find some operators $\{\widetilde D_j^{(1)}\}_{j\in\mathbb Z}\in \mathcal D$ and some natural numbers $d_1,\widetilde n_1$ with $d_1>2m$ such that $\|\widetilde D_r^{(1)}-P_m\|<\frac14$ and
\[
\|W_{r+\widetilde n_1d_1}\cdots W_{r+1}\widetilde D_r^{(1)}P_m\|<\frac14,
\]
\[
\|W_{r+1-\widetilde n_1d_1}^{-1}\cdots W_r^{-1}\widetilde D_r^{(1)}P_m\|<\frac14
\]
for each $r\in[J]$. Choose now some $\{\widetilde D_j^{(2)}\}_{j\in\mathbb Z}\in \mathcal D$ such that $\|\widetilde D_r^{(2)}-P_m\|<\frac{1}{4^2}$ for each $r\in[J]$. By the assumption (2), we can find some $d_2>2m$ and some $\widetilde n_2>\widetilde n_1d_1$ such that
\[
\|W_{r+\widetilde n_2 \widetilde d_2}\cdots W_{r+1}\widetilde D_r^{(2)}P_m\|<\frac{1}{4^2}
\]
and
\[
\|W_{r+1-\widetilde n_2d_2}^{-1}\cdots W_r^{-1}\widetilde D_r^{(2)}P_m\|<\frac{1}{4^2}
\]
for each $r\in[J]$. Set $n_1:=\widetilde n_1d_1$ and $n_2:=\widetilde n_2d_2$, then $n_2>n_1$. Also, set $D_r^{(1)}:=\widetilde D_r^{(1)}P_m$ and $D_r^{(2)}:=\widetilde D_r^{(2)}P_m$ for each $r\in[J]$, then
\[
\|D_r^{(i)}-P_m\|\leq\| D_r^{(i)}-P_m^2\|\leq\|\widetilde D_r^{(i)}-P_m\|<\frac{1}{4^i}
\]
for $i\in\{1,2\}$ and each $r\in[J]$.
Proceeding inductively, in this way we can construct a strictly increasing sequence $\{n_k\}_k\subseteq\mathbb N$ and sequences $\{D_r^{(k)}\}_{k\in\mathbb N}$ for each $r\in[J]$ such that
\[
\lim_{k\to\infty}\|D_j^{(k)}-P_m\|=0
\]
and
\[
\lim_{k\to\infty}\|W_{j+n_k}\cdots W_{j+1}D_j^{(k)}\|
=\lim_{k\to\infty}\|W_{j+1-n_k}^{-1}\cdots W_j^{-1}D_j^{(k)}\|=0
\]
for all $j\in[J]$. By \cite[Proposition 3.1]{FIL} it follows that $T$ is topologically transitive on $\mathcal A$, hence,  by combining this fact together with (1) which follows from (2), we deduce that $T$ is Devaney chaotic on $\mathcal A$.\\
$(4)\Rightarrow(1)$ This implication is trivial and follows from the definition of Devaney chaos. \end{proof}
	\begin{remark}\label{napomena-shift}
	Although, generalized bilateral weighted shift operators were considered in \cite{advances}, there are some important differences between the results in these papers and Theorem \ref{period-kompakt}. First, notice that the norm on $l_{2}(\mathcal{C})$  is different from the norm on $ \ell_2$-spaces considered in \cite{advances} because the norm on $l_{2}(\mathcal{C})$ is defined as $$\|\{ F_j\}_{j \in \mathbb{Z}} \|_2 = \| \sum_{j \in \mathbb{Z} } F_j^{*} F_j  \|^{ \frac{1}{2} }$$ for all $\{ F_j\}_{j \in \mathbb{Z}} \in l_{2}(\mathcal{C}) ,$ and \textit{not} as $$\|\{ F_j\}_{j \in \mathbb{Z}} \|_2 =  (\sum_{j \in \mathbb{Z} } \| F_j^{*} F_j  \|)^{ \frac{1}{2} }$$ for all $\{ F_j\}_{j \in \mathbb{Z}} \in l_{2}(\mathcal{C}).$ Moreover,  $l_{2}(\mathcal{C})$  consists of all those sequences $\{ F_j\}_{j \in \mathbb{Z}} \subseteq \mathcal C$ for which the infinite sum  $ \sum_{j \in \mathbb{Z} } F_j^{*} F_j $ is norm convergent in $\mathcal C$  and \textit{not} only of those $\{ F_j\}_{j \in \mathbb{Z}} \subseteq \mathcal C$ for which the infinite sum  $ \sum_{j \in \mathbb{Z} } \|F_j^{*} F_j\| $ is finite. These are the crucial differences between the set-up in \cite{advances} and the set-up in this section. 
\end{remark}
In what follows, we let $L$ be the operator on $\mathcal A=l_{2}(\mathcal{C})  $ given by $$ L (\lbrace F_j \rbrace_{j \in \mathbb Z}) \to  = \lbrace W_j F_{j-1} \rbrace_{j \in \mathbb Z} $$ for all $ \lbrace F_j \rbrace_{j \in \mathbb Z} \in \mathcal A,$ whereas we let $T$ be the operator considered in Theorem \ref{period-kompakt}, i.e. $$ T (\lbrace F_j \rbrace_{j \in \mathbb Z}) \to  = \lbrace W_j F_{j-1} \mathcal U \rbrace_{j \in \mathbb Z} $$ for all $ \lbrace F_j \rbrace_{j \in \mathbb Z} \in \mathcal A.$ Of course, here again $\{ W_{j} \}_{j \in \mathbb{Z}}$ is
a sequence of operators in $B(H)$ which is uniformly bounded in norm such that each $W_{j}$ has a bounded inverse and such that $\{ W_{j}^{-1} \}_{j \in \mathbb{Z}}$ is also uniformly bounded in the norm. \\
In the next example, we will show that it can happen that the operator $ T$ and the operator $L$ with same respective weight sequence $\{ W_{j} \}_{j \in \mathbb{Z}}$ are such that $T$ is Devaney chaotic on $ \mathcal A ,$ whereas $L$ is not Devaney chaotic on $ \mathcal A .$
	
\begin{example}\label{prvi-primer}
	
	Let $W_j=\sqrt{\frac{j}{j+1}}I$ if $j\in\mathbb N$ and $W_j=2I$ if $j\in\mathbb Z\setminus\mathbb N$. For given $J,m\in\mathbb N$, suppose that $\{D_i^{(k)}\}_k$ are sequences in $B_0(H)$ for $i\in [J]$ such that $$\lim_{k\to\infty}\|D_j^{(k)}-P_m\|=0$$ 
	for every $j\in\mathbb [J]$. Then we can find some $k_0\in\mathbb N$ such that $D_j^{(k_0)}\ne0$ for every $j\in [J]$. Hence for every $n_{k_0}\in\mathbb N$, we get,
	\[
	\sum_{\ell=1}^{\infty}D_1^{(k_0)^{*}}W_2^*\cdots W_{1+\ell n_{k_0}}^*W_{1+\ell n_{k_0}}\cdots W_2D_1^{(k_0)}
	= \sum_{\ell=1}^{\infty}\frac{2}{1+\ell n_{k_0}}D_1^{(k_0)*}D_1^{(k_0)} 
	\]
	\[ 
	\geq  \sum_{\ell=1}^{\infty}\frac{1}{\ell n_{k_0}} D_1^{(k_0)^{*}}D_1^{(k_0)}. 
	\]
Here the last inequality is taken in the cone of positive elements in $ \mathcal C. $	Since $ \sum_{\ell=1}^{\infty}\frac{1}{\ell n_{k_0}}=\infty $ for every $n_{k_0}\in\mathbb N$, we deduce that the series
	\[
	\sum_{\ell=1}^{\infty}D_1^{(k_0)^{*}}W_2\cdots W_{1+\ell n_{k_0}}^*W_{1+\ell n_{k_0}}\cdots W_2D_1^{(k_0)}
	\]
	cannot be convergent in $\mathcal C$, so the conditions of \cite[Theorem 3.3]{Nkoreja} are not satisfied. Hence, the operator $L$ cannot be Devaney chaotic on $ \mathcal A.$ However, it is straightforward to check that the conditions  of Theorem \ref{period-kompakt} is satisfied in this case, so the operator $T$ is Devaney chaotic on $ \mathcal A.$  \end{example}
Motivated by Example \ref{prvi-primer}, we will now continue with further comparison of chaotic behaviour of the operator $T$  and the operator $L.$
\begin{corollary}\label{poredjenje}
	Under the above notation and assumptions, assume that the operator $T$ and the operator $L$ have the same weight sequence $\{ W_{j} \}_{j \in \mathbb{Z}}.$ In this case, if $L$ is Devaney chaotic on $ \mathcal A,$ then $T$ is also Devaney chaotic on $ \mathcal A.$
\end{corollary}

\begin{proof}
	Let now 
	$
	\widetilde \Phi: A_1 \longrightarrow A_1
	$
	be given by
	\[
	\widetilde \Phi \bigl(\lbrace y_j \rbrace_{j\in\mathbb Z}\bigr)
	=
	\lbrace y_{j-1}\rbrace_{j\in\mathbb Z}.
	\]
	We want to show that $\widetilde \Phi$ is strongly aperiodic
	with respect to
	$
	\left\{\widetilde p_{(J,m)}\right\}_{J,m\in\mathbb N}.
	$ For this, it suffices to observe that for each
	$n,k,d\in\mathbb N$, with $d > 2J$, and all $ a=\lbrace F_{j} \rbrace_{j\in\mathbb Z} \in \mathcal{A}, $ we have
	$$
	\sum_{n=1}^{\infty}
	a \Phi^{ndk}\bigl(\widetilde p_{(J,m)}\bigr)=
	\sum_{n=1}^{\infty}
	P_m F_{j+ndk}^{*}F_{j+ndk}P_m=
	P_m
	\left(
	\sum_{n=1}^{\infty}
	F_{j+ndk}^{*}F_{j+ndk}
	\right)
	P_m, $$
	where in the last equality we use the convergence of the
	sum
	$
	\sum_{j=1}^{\infty}F_j^{*}F_j,
	$ in $ \mathcal C $
	and the fact that $d>2J$, so that
	\[
	\sum_{n=1}^{\infty}
	F_{j+ndk}^{*}F_{j+ndk}
	\leq
	\sum_{j=1}^{\infty}F_j^{*}F_j.
	\]
	
	Since we have that
	$
	\sum_{n=1}^{\infty}
	F_{j+ndk}^{*}F_{j+ndk}
	$
	is convergent for each $k\in\mathbb N$ and
	\[
	\lim_{k\to\infty}
	\sum_{n=1}^{\infty}
	F_{j+ndk}^{*}F_{j+ndk}
	=0,
	\]
	we deduce thus that 
	$\sum_{n=1}^{\infty}
	a \Phi^{ndk}\bigl(\widetilde p_{(J,m)}\bigr) $
	is convergent for each $k \in \mathbb{N}$ and
	$$ \lim_{k\to\infty}
	\sum_{n=1}^{\infty}
	a \Phi^{ndk}\bigl(\widetilde p_{(J,m)}\bigr)
	=0.$$
	Similarly, we can treat the series
	\[
	\sum_{n=1}^{\infty}
	a \Phi^{-ndk}\bigl(\widetilde p_{(J,m)}\bigr),
	\] to conclude that $\widetilde \Phi$ is strongly aperiodic with respect to
	$\{\widetilde p_{(J,m)}\}.$
	Thus, we can apply
	Theorem \ref{glavno-haos} to the operator $L.$
	So, if $L$  has a dense set of periodic elements,
	by Theorem \ref{glavno-haos} there exists a dense set $\mathcal{D}$ in $\mathcal{A}$
	such that for each $
	\{D_j\}_{j\in\mathbb Z}\in \mathcal{D} $
	and every $J,m\in\mathbb N$, there exists some
	$d\in\mathbb N$ such that the series
	\[
	\sum_{n=1}^{\infty}
	L^{ndk}	\bigl( \{D_j\}_{j\in\mathbb Z} \bigr)
	\Phi^{ndk}\bigl(\widetilde p_{(J,m)}\bigr)
	\]
	is convergent for each $k\in\mathbb N$. In particular, we can therefore choose
	$
	d>2J
	$ satisfying this assumption on the convergence of the respective series.
	Now we notice that
	\[
	\infty > \left\|
	\sum_{n=\tilde m}^{\infty}
	L^{nd}	\bigl( \{D_j\}_{j\in\mathbb Z} \bigr)
	\Phi^{nd}\bigl(\widetilde p_{(J,m)}\bigr)
	\right\|^2
	\]
	\[ 
	\geq  \left\| \sum_{n=\tilde m}^{\infty} 
	P_{m}D_{r}^{*}W_{r+1}^{*} \dots  W_{r+nd}^{*} W_{r+nd} \dots W_{r+1} D_{r} P_{m}
	\right\|
	\]
	for each $r\in[J]$ and all
	$ \widetilde m\in\mathbb N$, which follows by similar
	arguments as in the proof of Theorem \ref{period-kompakt}. Hence, for each $r \in [J],$ we obtain
	$$
	\lim_{\widetilde m\to\infty}
	\left\|
	W_{r+\widetilde m d}
	\dots W_{r+1}D_rP_m
	\right\|^2
	\leq
	\lim_{\widetilde m\to\infty}
	\left\|
	\sum_{n=\widetilde m}^{\infty}
	P_m D_r^{*}
	W_{r+1}^{*}
	\dots 
	W_{r+nd}^{*}
	W_{r+nd}
	\dots 
	W_{r+1}D_rP_m
	\right\|
	\\
	=0.
	$$
	
	Similarly, from the convergence of the series
	\[
	\sum_{n=1}^{\infty}
	L^{-nd}	\bigl( \{D_j\} \bigr)
	\Phi^{-nd}\bigl(\widetilde p_{(J,m)}\bigr)
	\]
	we deduce that
	\[
	\lim_{\widetilde m\to\infty}
	\left\|
	W_{r+1-\tilde m d}^{-1} \dots W_{r}^{-1}D_{r}P_{m}
	\right\|
	=0,
	\]
	for each $r\in[J]$. By Theorem \ref{period-kompakt}, it follows then that the operator $T$ is Devaney chaotic on $ \mathcal A.$
\end{proof}

In the next example, we will examine the case when all the weights $W_j$ are equal to one fixed bounded invertible linear operator $W$ on $H.$
		\begin{example}\label{jedan-operator} 
			 Let $\{w_j\}_{j\in\mathbb Z}\subseteq\mathbb R^+$ such that $w_j=\sqrt{\frac{j}{j+1}}$ if $j\in\mathbb N$, and $w_j=2$ else, and let $W\in B(H)$ such that $W(e_j)=w_je_{j+1}$ for all $j\in\mathbb Z$, so $W$ is the weighted shift w.r.t. $\{e_j\}_{j\in\mathbb Z}$. By similar arguments as in the previous example, it can be shown that $W$ is not Devaney chaotic on $H$. Let now
		\[
		\widetilde{ \mathcal D}:=\{P_mF,\ m\in\mathbb N,\ F\in \mathcal C\}.
		\]
		Since $\{P_m\}_{m\in\mathbb N}$ is also a right approximate unit for $ \mathcal C,$ which can easily be deduced by applying \cite[Proposition 2.2.1]{MT} and taking the adjoints, it follows that $ \widetilde{ \mathcal D}$ is dense in $ \mathcal C$. Then, the set
		\[
		\mathcal D:=\{\{D_j\}_{j\in\mathbb Z}\in\mathcal A\mid D_j\in \widetilde{ \mathcal D}\text{ for all }j\}
		\]
		is dense in $\mathcal A$.
		Now, given $J\in\mathbb N$ and $\{D_j\}_{j\in\mathbb Z}\in \mathcal D$, we can find some $M\in\mathbb N$ such that $P_MD_j=D_j$ for all $j\in[J]$. Hence,
		\begin{align*}
			\lim_{n\to\infty}\| D_jP_m\|&=\lim_{n\to\infty}\|W^nP_MD_jP_m\|\\
			&\leq\lim_{n\to\infty}\|W^nP_M\|\,\|D_jP_m\|=0
		\end{align*}
		and, similarly,
		\[
		\lim_{n\to\infty}\|W^{-n}D_jP_m\|=0
		\]
		for every $m\in\mathbb N$ and $j\in[J]$. 
		Suppose now that $W=W_r$ for all $r\in\mathbb Z$. Then the conditions of Theorem \ref{period-kompakt} are satisfied, so $T$ is Devaney chaotic on $\mathcal A$. However, $W$ itself is not Devaney chaotic on $H,$ which can easily be derived from \cite[Theorem 3.3]{Nkoreja} in the special case of $l_{2}(\mathbb Z).$ \end{example}
		
		In sequel, we let  $T_{W, \mathcal U} $ be the operator on $\mathcal C $ given by  $ T_{W, \mathcal U}(F)=\mathcal WF\mathcal U^*$ for all $F\in\mathcal C.$ We will now apply Theorem \ref{glavno-haos} to study the Devaney chaos for the operator  $T_{W, \mathcal U} .$ 
		 \begin{proposition}\label{elementarna-preslikavanja}
		 	Let $W$ be an invertible bounded linear operator on $H$ such that $W$ acts as a bilateral weighted forward shift with respect to the orthonormal basis  $ \{e_{j}\}_{j \in \mathbb Z}.$ Suppose that for every $r \in \mathbb Z $ and $\widetilde m \in \mathbb N $ it holds that $$ \lim_{r\to\infty}\|W^rP_{\widetilde m}\|=\lim_{r\to\infty}\|W^rP_{\widetilde m}\|=0 .$$ Then the operator $T_{W, \mathcal U} $ is Devaney chaotic on $ \mathcal C. $\\
		 	 In particular, there exist a bounded invertible linear operator $W$ on $H$ which is not Devaney chaotic on $H$ but such that the induced operator $T_{W, \mathcal U} $ is Devaney chaotic on $ \mathcal C. $
		 \end{proposition}
		\begin{proof}
			In this proof, for every $ n,d,m \in \mathbb N $ with $ d >2m $ we let $\widetilde Q_n^{(d, m)} $ denote the orthogonal projection onto $\operatorname{Span}
			\{e_{j}\}_{j=-m+nd}^{\,m+nd}. $ Also, we let  $\Phi(F)=\mathcal UF\mathcal U^*$ for all $F\in\mathcal C$ where $W$ is a bilateral weighted forward shift with respect to the orthonormal basis $ \{e_{j}\}_{j \in \mathbb Z}$ which satisfies the assumption in the proposition. Set $b:=W\mathcal U^*$. It follows easily  that $\{P_m\}_{m\in\mathbb N}$ has right approximation property in $\mathcal C$ and satisfies the condition (L). We want to show that $\Phi$ is strongly aperiodic with respect to $\{P_m\}_{m\in\mathbb N}$.
			To this end, given $\varepsilon>0, F\in\mathcal C$ and $m\in\mathbb N$, choose some $N\in\mathbb N$ such that $\|F-FP_n\|<\varepsilon$ for all $n\geq N$. Choose any $d\in\mathbb N$ with $d>2m,$ and for each $ n \in \mathbb N,$ let,  $\widetilde Q_n^{(d, m)} $ denote the orthogonal projection onto $\operatorname{Span}
			\{e_{j}\}_{j=-m+nd}^{\,m+nd}. $ For every $n_2>n_1>N+m+1$,
			\begin{align*}
				\left\|\sum_{n=n_1}^{n_2}F\Phi^{nd}(P_m)\right\|
				&=\left\|\sum_{n=n_1}^{n_2}F\mathcal UP_m\mathcal U^{*nd}\right\|\\
				&=\left\|\sum_{n=n_1}^{n_2}F\widetilde Q_n^{(d)}\right\|
				=\left\|F\sum_{n=n_1}^{n_2}\widetilde Q_n^{(d)}\right\|.
			\end{align*}
			By similar arguments as in the proof of Lemma \ref{aperiodic}, we can deduce that, since $d>2m$, then $\widetilde Q_n^{(d,m)^{\prime}}-s$ are mutually orthogonal and
			\[
			\sum_{n=n_1}^{n_2} \widetilde Q_n^{(d,m)}\leq I-P_N
			\]
			because $-m+nd\geq -m+(N+m+1)d\geq -m+N+m=N$ for all $n>N+m+1$. Hence
			\begin{align*}
				\left\|F\sum_{n=n_1}^{n_2}\widetilde Q_n^{(d,m)}\right\|
				&=\left\|F(I-P_N)\sum_{n=n_1}^{n_2}\widetilde Q_n^{(d,m )}\right\|\\
				&\leq\|F(I-P_N)\|\left\|\sum_{n=n_1}^{n_2}\widetilde Q_n^{(d,m)}\right\|\\
				&=\|F(I-P_N)\|<\varepsilon
			\end{align*}
			because $\left\|\sum_{n=n_1}^{n_2}\widetilde Q_n^{(d,m)}\right\|=1$ since $\sum_{n=n_1}^{n_2}\widetilde Q_n^{(d,m)}$ is an orthogonal projection. Since $\varepsilon>0$ was chosen arbitrarily, we can conclude that $\sum_{n=1}^{\infty}F\Phi^{nd}(P_m)$ is convergent. Further, since for every $k\in\mathbb N$ with $k>N+m+1$, we have $-m+nkd\geq N$, by similar arguments as above we can show that
			\[
			\left\|\sum_{n=1}^{\ell}F\Phi^{ndk}(P_m)\right\|<\varepsilon
			\]
			for all $\ell\in\mathbb N$ and each $k\in\mathbb N$ with $k>N+m+1$ (with $\varepsilon>0$ and $N\in\mathbb N$ being the same as above).
			Hence
			\[
			\left\|\sum_{n=1}^{\infty}F\Phi^{nkd}(P_m)\right\|\leq\varepsilon
			\]
			for each $k>N+m+1$ . Again, since $\varepsilon>0$ was chosen arbitrarily, we can conclude that
			\[
			\lim_{k\to\infty}\left\|\sum_{n=1}^{\infty}F\Phi^{ndk}(P_m)\right\|=0.
			\]
			The proof for the case of series $$\sum_{n=1}^{\infty}F\Phi^{-ndk}(P_m)$$ 
			is similar. Next, we observe that for each $F\in\mathcal C$ and $n,m\in\mathbb N$,
			\[
			T_{W,\mathcal U}^n(F)\Phi^n(P_m)=W^nFP_m \mathcal U^{*n }.
			\]
			Also, we recall from Example \ref{jedan-operator} that the set
			\[
			\widetilde{ \mathcal D}:=\{P_nG\mid n\in\mathbb N,\ G\in\mathcal C\}
			\]
			is dense in $\mathcal C$.
			Let $\widetilde F\in \widetilde{ \mathcal D}$ and $m\in\mathbb N$ be given. Then there exists some $\widetilde m\in\mathbb N$ such that $P_{\widetilde m}\widetilde F=\widetilde F$. Choose some $d\in\mathbb N$ with $d>2(m+\widetilde m)$. Then $\{\widetilde Q_n^{(d,m+\widetilde m)}\}_n$ is a sequence of mutually orthogonal projections. Notice also that since $\{-m+nd,\ldots,m+nd\}$ and $\{-\widetilde m+nd,\ldots,\widetilde m+nd\}$ are both subsets of $\{-(m+\widetilde m)-nd,\ldots,m+\widetilde m-nd\}$, then
			\[
			\widetilde Q_n^{(d,m)}\widetilde Q_n^{(d,m+\widetilde m)}=\widetilde Q_n^{(d,m)}
			\]
			and
			\[
			\widetilde Q_n^{(d,m+\widetilde m)}\widetilde Q_n^{(d,\widetilde m)}=\widetilde Q_n^{(d,\widetilde m)}
			\]
			for all $n\in\mathbb N$. Further, since $W$ is a weighted bilateral forward shift with respect to $\{e_j\}_{j\in\mathbb Z}$, then
			$$
			W^{ndk}P_{\widetilde m}=\widetilde Q_{nk}^{(d,\widetilde m)}W^{ndk}P_{\widetilde m}
			=\widetilde Q_{nk}^{(d,m+\widetilde m)}\widetilde Q_{nk}^{(d,\widetilde m)}W^{ndk}P_{\widetilde m}
			$$
			for all $n,k\in\mathbb N$.
			Similarly,
			\begin{align*}
				P_m \mathcal U^{*ndk}&=P_m \mathcal U^{*ndk}\widetilde Q_{nk}^{(d,m)}\\
				&=P_m \mathcal U^{*ndk}\widetilde Q_{nk}^{(d,m)}\widetilde Q_{nk}^{(d,m+\widetilde m)}
			\end{align*}
			for all $n,k\in\mathbb N$.
			Hence, by utilizing all these facts together with Lemma \ref{pomocno}, we obtain for all $n_1,n_2,k\in\mathbb N$ with $n_2>n_1$,
			\begin{align*}
				&\left\|\sum_{n=n_1}^{n_2}T_{W, \mathcal U}^{ndk}(\widetilde F)\Phi^{ndk}(P_m)\right\|\\
				&=\left\|\sum_{n=n_1}^{n_2}\widetilde Q_{nk}^{(d,m+\widetilde m)}\widetilde Q_{nk}^{(d,\widetilde m)}W^{ndk}P_{\widetilde m}\widetilde FP_m \mathcal U^{*ndk}\widetilde Q_{nk}^{(d,m)}\widetilde Q_{nk}^{(d,m+\widetilde m)}\right\|\\
				&=\sup_{n_1\leq n\leq n_2}\left\|\widetilde Q_{nk}^{(d,m+\widetilde m)}\widetilde Q_{nk}^{(d,\widetilde m)}W^{ndk}P_{\widetilde m}\widetilde FP_m \mathcal U^{*ndk}\widetilde Q_{nk}^{(d,m)}\widetilde Q_{nk}^{(d,m+\widetilde m)}\right\|,
			\end{align*}
			where in the last equality we apply Lemma \ref{pomocno} together with the fact that $\{\widetilde Q_{nk}^{(d,m+\widetilde m)}\}_{n\in\mathbb N}$ is a sequence of mutually orthogonal projections for each $k\in\mathbb N$. However, for each $n,k\in\mathbb N$, we have
			\begin{align*}
				&\left\|\widetilde Q_{nk}^{(d,m+\widetilde m)}\widetilde Q_{nk}^{(d,\widetilde m)}W^{ndk}P_{\widetilde m}\widetilde FP_m \mathcal U^{*ndk}\widetilde Q_{nk}^{(d,m)}\widetilde Q_{nk}^{(d,m+\widetilde m)}\right\|\\
				&\leq\|W^{ndk}P_{\widetilde m}\|\,\|\widetilde F\|,
			\end{align*}
			hence by using that $\lim_{r\to\infty}\|W^rP_{\widetilde m}\|=0$, we easily obtain by all the above that the series
			\[
			\sum_{n=1}^{\infty}T_{W, \mathcal U}^{ndk}(\widetilde F)\Phi^{ndk}(P_m)
			\]
			converges for all $k\in\mathbb N$, and
			\[
			\lim_{k\to\infty}\sum_{n=1}^{\infty}T_{W, \mathcal U}^{ndk}(\widetilde F)\Phi^{ndk}(P_m)=0.
			\]
			The proof for the case of series $$\sum_{n=1}^{\infty}T_{W, \mathcal U}^{-ndk}(\widetilde F)\Phi^{-ndk}(P_m)$$ 
			is similar. Since the conditions of Theorem \ref{glavno-haos} are satisfied, we deduce that the operator $ T_{W, \mathcal U}$ is Devaney chaotic on $ \mathcal C. $\\
			In order to prove the last statement of the proposition, let $W $ be the operator from Example \ref{jedan-operator}. By a direct calculation it can be checked that $W$ satisfies the assumptions of the proposition, i.e. for every $r \in \mathbb Z $ and $\widetilde m \in \mathbb N $ it holds that $$ \lim_{r\to\infty}\|W^rP_{\widetilde m}\|=\lim_{r\to\infty}\|W^rP_{\widetilde m}\|=0 .$$ By what we have proved above, it follows that the induced operator $T_{W, \mathcal U} $ is Devaney chaotic on $ \mathcal C.$ However, as we argued in Example \ref{jedan-operator}, $W$ itself is not Devaney chaotic on $H.$
		\end{proof}
	\begin{remark}\label{veza-Shiraz}
		Some sufficient conditions for the Devaney chaos of the operator $ T_{W, \mathcal U}$ on $ \mathcal C $ have already been given in \cite[Theorem 5]{iv2}. However, the operator $W$ from Example \ref{jedan-operator} does not satisfy these conditions, which is not hard to verify by direct calculation. Therefore, we can conclude that the sufficient conditions of Proposition \ref{elementarna-preslikavanja} are different from those given in \cite[Theorem 5]{iv2}.
	\end{remark}
	We end this paper with the following question.
	
	\textbf{Question:} Does the Devaney chaos of $W$ on $H$ imply the Devaney chaos of $T_{W, \mathcal U} $  on $ \mathcal C$ ?

		\bibliographystyle{amsplain}

\end{document}